\documentclass[10pt,a4paper]{article}
\usepackage{algorithmic}
\usepackage {graphicx}
\usepackage[latin1]{inputenc} 
\usepackage {url}
\usepackage{a4}
\usepackage{latexsym}
\usepackage{amsmath}
\usepackage{amssymb}
\usepackage{enumerate}
\usepackage{cases}
\usepackage{multirow}
\usepackage{multicol}
\usepackage{bbold}
\usepackage{hyperref}
\usepackage{array}
\usepackage{empheq}
\usepackage{float}
\usepackage{tikz}
\usepackage{subcaption}
\usepackage{booktabs}

\author{Zacharie \textsc{Ales}$^{\dag,\ddag}$ \and Sourour \textsc{Elloumi}$^{\dag,\ddag}$}
\date{
    $^\dag$ UMA, ENSTA Paris, Institut Polytechnique de Paris, 91120 Palaiseau, France\\
    $^\ddag$ CEDRIC, CNAM, Paris, France\\[2ex]    
    \{firstname.lastname\}@ensta-paris.fr\\[2ex]
}
\definecolor{lighttext}{rgb}{0,0,0}

\graphicspath{{img}}
\usetikzlibrary{arrows,shapes,decorations,automata,backgrounds,petri,positioning,decorations.markings}
\hypersetup{
    colorlinks,
    citecolor=lighttext,
    filecolor=lighttext,
    linkcolor=lighttext,
    urlcolor=lighttext
}

\usepackage{amsthm}
\usepackage{subcaption}
\newtheorem{theorem}{Theorem}[section]
\newtheorem{definition}{Definition}[section]
\newtheorem{corollary}{Corollary}[section]
\newtheorem{property}{Property}[section]

\newcommand{\nom}[1]{\overline{#1}}
\newcommand{\rea}{^{r}}
\newcommand{\pro}{^{p}}
\newcommand{\SetStations}{\mathcal V }

\title {A solution robustness approach applied to network optimization problems}

\newcommand{\NP}{\mathcal N\mathcal P}
\newcommand{\s}[1]{\{#1\}}
\renewcommand{\arraystretch}{1.4}

\definecolor{lightbg}{rgb}{1, 1, 1}
\pagecolor{lightbg}

\makeatletter
\newcommand{\globalcolor}[1]{%
  \color{#1}\global\let\default@color\current@color
}
\makeatother

\AtBeginDocument{\globalcolor{lighttext}}

\newcommand{\zcounter}[1]{\addtocounter{equation}{1}\newcounter{#1}\setcounter{#1}{\arabic{equation}}\zcount{#1}}
\newcommand{\zcount}[1]{(\arabic{#1})}

\begin{document}
\vspace{-4cm}
\maketitle

\hrule
\vspace{.1cm}
\begin{small}

\noindent\textbf{Abstract}
\vspace{.15cm}

\textit{Solution robustness} focuses on structural similarities between the nominal solution and the scenario solutions. Most other robust optimization approaches focus on the \textit{quality robustness} and only evaluate the relevance of their solutions through the objective function value.  However,  it can be more important to optimize the {solution robustness} and, once the uncertainty is revealed, find an alternative scenario solution $x^s$ which is as similar as possible to the nominal solution $x^{nom}$. This for example occurs when the robust solution is implemented on a regular basis or when the uncertainty is revealed late. We call this distance between $x^{nom}$ and $x^s$  the \textit{solution cost}.  We consider the \textit{proactive problem} which minimizes the average solution cost over a discrete set of scenarios while ensuring the optimality of the nominal objective of $x^{nom}$.

We show for two different solution distances $d_{val}$ and $d_{struct}$ that the proactive problem is NP-hard for both the integer min-cost flow problem with uncertain arc demands and for the integer max-flow problem with uncertain arc capacities. For these two problems, we prove  that  once the uncertainty is revealed, even identifying a \textit{reactive solution} $x^r$ with a minimal distance to a given solution $x^{nom}$ is NP-hard for $d_{struct}$, and that it is polynomial for $d_{val}$. 

{We highlight the benefits of solution robustness in a case study on a railroad planning problem. First, we compare our  proactive approach to the anchored and the $k$-distance approaches. Then, we show the efficiency of the proactive solution over reactive solutions. Finally, we illustrate the solution cost reduction when relaxing the optimality constraint on the nominal objective of the proactive solution $x^{nom}$. }
\vspace{.3cm}

\noindent\textit{Keywords: } Robust optimization; Uncertainty modelling; Solution robustness; NP-hardness; Proactive problem; Mixed-Integer Linear Programming.
\end{small}
\vspace{.3cm}

\hrule
\vspace{.3cm}

\section{Introduction}

In many industrial applications a \textit{nominal solution} ${x^{nom}}$ of an
optimisation problem $P$ is intended to be used on a regular basis. This is for
example the case in railway scheduling where a single planning can be repeated
each week.

Due to uncertainties on the data of $P$, ${x^{nom}}$ may occasionally become
infeasible and, thus, require to be adapted. {The \textit{quality robustness} of $x^{nom}$ characterizes how low the deterioration of the objective function of $P$ is when changes occur. }

Robust and stochastic methods tend to optimize the quality robustness {which may} not be the most relevant
approach. Indeed, since solution $x^{nom}$ is regularly implemented, people involved in its
realisation (e.g, employees, clients, ...) are accustomed to it. As a consequence, they may,
by habit, disrupt the realisation of the adapted solution $x^{s}$. For example, users may miss a
train if it leaves earlier than usual or an operator may activate a railroad
switch which is no longer required. These disruptions may lead to
incidents and customers dissatisfactions or even the need to adapt the solution once again. We call
 {\textit{solution costs}} {such indirect costs} that come from the adaptation of ${x^{nom}}$ into
$x^{s}$. {The lower the solution costs are, the higher the \textit{solution robustness} is~\cite{sevaux2004genetic,vandevonder2005use}.}

In this paper, we {introduce a solution robustness approach} which minimizes the solution cost  while limiting the
 nominal objective of the nominal solution. We modelize the
solution cost between two solutions by  a distance. The key idea is that the closer the solutions are to each other, the less likely
the human errors are.

In Section~\ref{sec:struct} we formally present {our solution robustness approach} and introduce the nominal, the reactive and the proactive problems. We discuss  links with existing approaches in Section~\ref{sec:literature}. {Sections~\ref{sec:mcf} and Section~\ref{sec:mf} are}
dedicated to  two network problems namely, the min-cost
flow problem and the max-flow problem within the framework of solution
robustness. We caracterize the complexity of the reactive and the proactive problems {for two different solution distances}. {Eventually, we present a case study on a railroad planning problem in Section~\ref{sec:experiments}, highlighting the benefits of the proactive approach over the reactive approach. We also compare ourselves to two other approaches from the litterature.}

\section{Reactive and proactive robustness solutions}
\label{sec:struct}

{We consider the deterministic version of an optimisation problem $\nom{P}$. We call this version the \textit{nominal problem}. Its data correspond to the normal situation:}

\begin{center}
$\nom{P}\left\{	
\begin{array}{lll}
\min & {f(x)}\\	
\mbox{s.t.} & {x\in X}
\end{array}
\right.$
\end{center}

{\noindent where $X\subset\mathbb R^n$ is the feasible solution set of $\overline P$ {and its objective function $f:\mathbb R^n\mapsto\mathbb R$ is called the \textit{nominal objective}.}} 

We modelize the uncertainty {on the data by a set of scenarios which may occur
$\mathcal S=\s{\xi_i}_{i=1}^{|\mathcal S|}$}. To each scenario $\xi_i$ 
is associated {a feasible set $X_i$} and a positive weight $w_i\in\mathbb R^+$ which represents its importance.
The weight of a scenario can for example be chosen according to its probability
of occurrence or its gravity.

{We now introduce two problems which represent two possible approaches to handle the uncertainty. In the \textit{reactive problem}, the uncertainty has not been anticipated and a feasible solution is obtained once the uncertainty is revealed. In the \textit{proactive problem}, a nominal solution and one solution for each scenario are computed a priori.}

\subsection{Reactive robust solutions}
\label{sec:reac}

We suppose that a solution ${x^{nom}}$ of the nominal problem $\nom P$ is known and
that scenario $\xi_i$ occurs at operating time (i.e., a few days or
hours before its realisation). The adaptation of ${x^{nom}}$ to scenario~$\xi_i$ may
not be a trivial problem. This is all the more true when many resources are
required to the concrete implementation of a solution. For example, in railway
scheduling, different services must communicate to synchronize their ressources
and trains may have to be transfered in advance between stations. Consequently,
the solution cost of the new solution may become more crucial than optimizing its nominal objective and ignoring it may even lead to infeasibilities. Indeed, there may
simply not be enough time to implement drastic modifications that would
minimize the nominal objective but would require the coordination of multiple
services and resources. Moreover, modifications of a recurrent solution at
operating time increases the risks of human errors. 

This motivates us to define a \textit{reactive problem} which provides a \textit{reactive solution} $x^r$ satisfying feasible set $X_i$ of $\xi_i$ and which distance to $x^{nom}$ is minimal:

\begin{center}
	${P\rea(\xi_i, {x^{nom}})}\left\{	
	\begin{array}{lll}
	\min & d(x\rea, {x^{nom}})\\	
	\mbox{s.t.} & {x\rea\in X_i}
	\end{array}
	\right.$
\end{center}

\noindent where $d:\mathbb R^n\times\mathbb R^n\mapsto\mathbb R$ is a distance function
between $x\rea$ and ${x^{nom}}$.

The reactive problem is considered once the uncertainty is revealed. We now
introduce the proactive problem which allows to anticipate the uncertainties.

\subsection{Proactive robust solutions}
\label{sec:proac}

A reactive solution $x^r$ may have a high nominal objective value $f(x^r)$ as $f$ is not taken into account in the reactive problem.
Furthermore, the solution cost may also be consequent as it has not been
anticipated. To address these issues, we introduce the \textit{proactive
problem}  $P\pro$ which variables are a \textit{proactive solution} $x\pro$ of the nominal
problem as well as an \textit{adapted solution} $x^i$ for each scenario $\xi_i\in\mathcal S$:

\begin{empheq}[left={P\pro(\mathcal S\hbox{,\,} c^*)}\empheqlbrace]{alignat=2}
	\min~ & \sum_{\xi_i\in\mathcal S} w_i\,d(x\pro, x^i) \label{eq:pro_objective}\\	
	\mbox{s.t.~} & x\pro\in X \label{eq:pro_A}\\
	& x^{i}\in X_i & \forall \xi_i\in\mathcal S\label{eq:pro_s}\\
	& f(x^p) = c^*\label{eq:pro_price}
\end{empheq}

\noindent where $c^*$ is the optimal value of the nominal problem.

Objective~\eqref{eq:pro_objective} minimizes the weighted sum of the solution
costs over all the scenarios. Constraints~\eqref{eq:pro_A} and~\eqref{eq:pro_s}
ensure that $x\pro$ is feasible for $(\nom P)$ and that $x^i$ is feasible for
scenario $\xi_i$. Eventually, Constraint~\eqref{eq:pro_price} ensures that the nominal objective of $x\pro$ is equal to $c^*$. Consequently, $x\pro$ is an optimal solution of the nominal problem which additionally minimizes the weighted sum of the solution costs
over $\mathcal S$.

Since the weights $w_i$ are positive in problem $P\pro(\mathcal S, c^*)$, each solution $x^i$ is an optimal solution of $P^r(\xi_i, x^p)$. As a consequence, if the nominal problem has a unique optimal solution $x^*$ then, solving $P^p(\mathcal S, c^*)$ is equivalent to solving $\overline P$ and $P^r(\xi_i, x^*)$ for $i\in\s{1, ..., |\mathcal S|}$.

Note that the proactive problem $P^p$ requires to know the optimal value $c^*$ of an optimal solution of
the nominal problem. This should generally not be a problem as $\nom{P}$ is
likely to be significantly easier to solve than $P^p$ given that $\nom{P}$
only finds one solution $\nom x$ while $P^p$ simultaneously finds $|\mathcal
S|+1$ solutions. 

{The proactive problem finds an optimal solution of the nominal problem. However,  the solution cost may be further improved by allowing a given flexibility on the nominal objective of the proactive solution $x\pro$. To allow such a flexibility, Constraint~\eqref{eq:pro_price} can be relaxed into the following constraint:}

	\begin{equation}
		f(x\pro) \leq c^* (1+\varepsilon),
		\label{eq:pro_price_eps}
	\end{equation}

 \noindent{where $\varepsilon\in \mathbb R^+$ represents the maximal allowed percentage of
 increase of the nominal objective value of $x\pro$. }


 {We now introduce two distances for the reactive and proactive problems.}

\subsection{Defining solution costs with two distances}

A
first intuitive distance corresponds {the $\ell_1$ norm} that we call the
\textit{distance in values}.

\begin{definition}
	\label{def:value}
	{\textit{Distance in values}} $d_{val}(x^1,x^2)\triangleq\displaystyle\sum_{j=1}^n |x^1_j - x^2_j|$.
\end{definition}
	
Depending on the context, $d_{val}$ may not be the most relevant distance. In
railway scheduling for example,  increasing by $1$ the number of
units on an existing train is much cheaper than creating a new
train. Indeed, the creation of a train requires to check the availability of a
locomotive and of agents as well as the compatibility of the train schedule
with that of other trains. Thus, we also introduce a \textit{distance in 
structure}. Let $\mathbb 1_{b}$ be equal to $1$ if condition $b$ is true and
$0$ otherwise.
	
\begin{definition} \label{def:structure} {\textit{Distance in structure}}
$d_{struct}(x^1, x^2)\triangleq\displaystyle\sum_{j=1}^n |\mathbb 1_{x^1_j > 0} - \mathbb 1_{x^2_j
> 0}|$. \end{definition}

{Note that some problems may contain several subsets of variables. In that case, it may be relevant to apply these distances to some subsets only. }
	
\subsection{Literature review on solution cost in robust optimization}
\label{sec:literature}

 
In the line of~\cite{soyster1973convex}, robust optimization approaches often search a nominal solution which is feasible for any value of the data in an uncertainty set~\cite{ben1998robust,bertsimas2004price}. In this context, the solution cost is equal to $0$ as the solution is not modified once the uncertainty is revealed and the \textit{price of robustness}~\cite{bertsimas2004price} corresponds to the minimal increase of the nominal objective required to ensure the feasibility of a single solution in all possible scenarios.

In many concrete problems a competitive nominal solution may not be feasible for all possible realisations of the uncertainty. To alleviate this problem, two-stage robust optimization approaches have been introduced where the value of a set of recourse variables is fixed only once the uncertainty is revealed~\cite{atamturk2007two,ben2004adjustable}. To ease the resolution of such problems, restrictions are often imposed on the recourse variables~\cite{ben2004adjustable}. In adjustable robust problems, the value of the recourse variables can even be directly deduced from the value of the uncertain data~\cite{ben2003new,gupta2001provisioning,ouorou2007model,poss2013affine}, thus significantly reducing the complexity of the problem. In two-stage robust optimization approaches, the nominal objective is generally optimized and the solution cost is indirectly considered by the choice of the recourse variables. The article~\cite{bendotti2019anchor} constitutes an exception as the authors consider a scheduling problem in which the nominal objective is bounded and where the weighted sum of anchored jobs is maximized. A job is said to be \textit{anchored} if it can not be rescheduled once the uncertainty is revealed. This objective can be viewed as a way to limit the solution cost. {The authors consider uncertain processing times and show that maximizing the weight of the anchored jobs is polynomial for the box uncertainty set. It is however NP-hard for both the budgeted and the polytope uncertainty sets.}

Liebchen et al. introduced a general two-stage framework, called \textit{recoverable robustness}, in which a solution $x$ and a recovery algorithm $A$ are determined such that the nominal objective is optimized and, once the uncertainty is revealed, it is guaranteed that the application of $A$ on $x$ leads to a feasible solution~\cite{liebchen2009concept}. Restrictions on $A$ can be imposed to limit its recovery actions, its complexity or even the distance between $x$ and the recovered solution. This last constraint could be used to bound the solution cost. However, it is generally not easy to ensure that these restrictions are satisfied for any realization of the uncertainty.  In the solution robustness framework that we define in Section~\ref{sec:proac}, the proactive problem is able to minimize the solution cost at the price of an additional set of variables for each scenario considered. 

{Two examples of recoverable robustness applied to shortest path problems are introduced in~\cite{busing2012recoverable}. In \textit{$k$-distance recoverable robustness}, the recovery actions are limited since at most $k$ new arcs can be used once the uncertainty is revealed. In the second approach, called \textit{rent recoverable robustness}, an edge is said to be \textit{rented} if it is used in the first stage and it is \textit{bought} if it is used in the second stage. For each scenario $s$ and each arc $e$  the cost $c^s_{e}$ incurred for both renting and buying arc $e$ is defined. The cost of the arc is lower if it is only rented  ($\alpha c^s_{e}$ with $\alpha\in]0, 1[$) and higher if it is only bought ($(1+\beta)c^s_{e}$ with $\beta\geq 0$). This second approach is more flexible than the first one since the number of new edges used in the second stage is not constrained but it is also less generic. Indeed, the notion of renting and buying edges is relevant in problems such as railway scheduling where the network is owned by a company and exploited by others but it is not suitable for all applications. The authors study the complexity of both variants on three uncertainty sets and show that only the rent recoverable robustness on the interval set is polynomial. In Section~\ref{sec:experiments} we compare the $k$-distance recoverable robustness and the anchored approach to our proactive approach.}

{Another type of recoverable robustness called \textit{recovery-to-optimality} has been introduced \cite{goerigk2014recovery}. In this context, when a scenario $s$ occurs, the nominal solution $x^{nom}$ must be adapted into a solution $x^s$, with an optimal nominal objective and which, as a second objective, minimizes a recovery cost $d(x^{nom}, x^s)$. The objective is then to find a nominal solution which minimizes the average or the worst recovery cost over all the scenarios. The recovery cost can correspond to a solution cost or a regret. In addition to imposing the optimality of $x^s$, this approach is different from ours as it does not necessarily impose the feasibility of the nominal solution $x^{nom}$. Consequently, recovery-to-optimality appears more relevant when the solution cost is less important than the  nominal objective value as $x^{nom}$ may require a significant solution cost in order to be transformed into a solution $x^s$  with an optimal nominal objective for scenario $s$. This is for example the case when the uncertainty is revealed early (e.g., months or years before implementation) or more generally when the changes are unlikely to lead to errors (e.g., few humans interventions). On the contrary, our approach is more appropriate when the solution cost is at least as important as  the nominal objective value. For example, this occurs when the uncertainty is revealed late (e.g., a few minutes or hours before implementation) or when the cost of any error is high (e.g., when lives are at stake). The authors apply their approach to a linear program for timetabling in public transportation. Recovery-to-optimality requires the optimality of $x^s$ for each scenario $s$ which is harder to ensure than the optimality of $x^{nom}$ required by our approach. As a consequence, they consider a heuristic in which optimal solutions of a subset of scenarios $\mathcal S$ is first computed and then a solution $x$ which minimizes the recovery cost to these solutions is obtained.}

The notion of $(a, b)$-supermodel introduced in~\cite{ginsberg1998supermodels}  is related to the solution cost.  A solution $x$ is said to be an \textit{$(a, b)$-supermodel} if whenever the value of at most $a$ variables in $x$ is changed, then a feasible solution of the original problem can still be obtained by modifying the value of at most $b$ other variables. The authors show that a $(1, 1)$-supermodel of SAT  can be obtained by solving a larger instance of SAT. The originality of this approach is the fact that the uncertainty is not on the data but on the value of the solution itself.

{The identification of efficient schedules both in terms of nominal objective and solution cost has been considered under different names: predictability~\cite{mehta1998predictable,o1999predictable}, stability~\cite{herroelen2004construction,leus2003generation,leus2004branch,leus2005complexity}, solution robustness~\cite{sevaux2004genetic,vandevonder2005use}. Exact approaches are rarely considered to solve such problems. One exception is~\cite{leus2004branch} in which  a hand-made branch-and-bound algorithm is considered to solve a single-machine scheduling problem in which a single job is anticipated to be disrupted. In other works, heuristic approaches are generally considered. For example,} in~\cite{sevaux2004genetic} the authors present a multi-objective  genetic algorithm in which both the solution cost and the nominal objective are optimized. Their problem share features with both the reactive problem $P^r$ (defined in Section~\ref{sec:reac}) and the proactive problems $P^p$ (defined in Section~\ref{sec:proac}). Indeed, as in $P^r$, a nominal solution $x^{nom}$ is fixed and we search a solution $x$ which distance to $x^{nom}$ is minimized, and similarly to $P^p$, the cost of $x$ for several scenarios is additionally minimized. 

The approach presented in this paper has already been applied to a scheduling problem of SNCF, the french national railway company~\cite{lucas2020planification,lucas2019reducing}. We below apply this framework to two network problems, namely: the integer min-cost flow and the integer max-flow. It is well known that the deterministic version of both of these problems can be solved with polynomial time
algorithms~\cite{ahuja1988network} {but we show that it is not the case for their proactive counterparts}. We consider uncertain arc demands for the min-cost flow problem and uncertain arc capacities for the max-flow problem.


\section{Solution robustness for the min-cost flow problem with  uncertain arc demands}
\label{sec:mcf}

The min-cost flow problem can be stated as: 
\vspace{.3cm}

\noindent\fbox{
	\parbox{0.97\textwidth}{\textsc{Min-Cost Flow Problem} \newline
		\begin{tabular}[\textwidth]{p{0.06\textwidth}p{0.85\textwidth}}

Input: & A digraph $G = (V, A)$ with arc demands $\ell_a\in\mathbb N$, capacities  $u_a\in\mathbb N$ and unitary cost $c_a\in\mathbb R^+$ of each arc $a\in A$ and node demands $b_v\in\mathbb Z$
for each vertex $v\in V$ ($b_v>0$ if $v$ is a \textit{supply node}, $b_v <0$ if $v$ is a \textit{demand node} and $b_v=0$ if $v$ is a \textit{transshipment node}).\\ 

Output: & Find an integer flow with minimal cost.
		\end{tabular}
	}
}
\vspace{.3cm}

We assume that the uncertainties are on the arcs
demands $\ell$. We denote by $\ell^\xi$ the arc demands
associated to a scenario $\xi\in \mathcal S$.


We   characterize the complexity of the four  reactive and proactive min-cost flow problems associated with distances $d_{val}$ and $d_{struct}$, as represented in
Table~\ref{tab:complexity}. Note that all these problems are in $\NP$.

\begin{table}[H]
    \centering
    \renewcommand{\arraystretch}{1.4}

	\begin{tabular}{ccc}
		\hline
			 \multirow{2}{*}{\textbf{Problem}}& \multicolumn{2}{c}{\textbf{Distance}}\\
			 & $d_{val}$ & $d_{struct}$ \\\hline
			 \multirow{2}{*}{\textbf{Proactive}} & \multirow{2}{*}{$\NP$-hard (Section~\ref{sec:mcf_dval})} & $\NP$-hard (Section~\ref{sec:mcf_dstruct})\\
			 & & \small{even with $1$ scenario}\\\hline
			 {\textbf{Reactive}} & {Polynomial (Section~\ref{sec:mcf_dval})} & $\NP$-hard (Section~\ref{sec:mcf_dstruct})\\\hline
		\end{tabular}
    \caption{{Complexity of the min-cost flow problems with uncertain arc demands associated with distances $d_{val}$ and $d_{struct}$.}}
    \label{tab:complexity}
\end{table}

\subsection{Complexity of the robust min-cost flow with distance $d_{val}$}
	\label{sec:mcf_dval}

We show in this section that for distance $d_{val}$ the proactive problem is $\NP$-hard and that the reactive problem is polynomial.

	\begin{theorem}
		$MCF_{d_{val}}^p$ is $\NP$-hard.
		\label{th:MCF_DV_P}
	  \end{theorem}
	
	\begin{proof}
	
		{We prove this result with a reduction from problem $3$-SAT. This problem
	considers a boolean expression composed of $n$ variables $X=\s{x_1, ..., x_n}$
	and $m$ clauses $\s{C_1, ..., C_m}$. Let a \textit{literal} be either a boolean
	variable $x_i$ or its negation $\overline x_i$. Each clause of a $3$-SAT
	problem is a conjunction of three literals (e.g. $C_1=(x_1\vee\overline x_2\vee
	x_3)$). The aim of this problem is to determine whether there exists an
	assignment of the variables which satisfies all the clauses.}
	
	{We  now  present  how  to  construct  an instance  $I_{MCF}$  of  the optimization problem
	$MCF_{d_{val}}^p$  from an  instance  $I_{SAT}$  of the  feasibility problem $3$-SAT. Instance $I_{MCF}$ has $m$ scenarios. We prove that an optimal solution of $I_{MCF}$ leads to a solution cost of value $4mn$ if and only if $I_{SAT}$ is a yes-instance. 

\vspace{.3cm}
\noindent{\textbf{Construction of $I_{MCF}$}}

	As represented in Figure~\ref{fig:MCF_DV_P_G}, the set of nodes $V$ of the constructed graph $G=(V, A)$ is composed of:}
	\begin{itemize}
		\item {one source node  $s$ {with supply} $b_s=n$ and one sink node $t$ {with demand} $b_s=-n$;}
		\item for each boolean variable $x_i$ one \textit{variable node} $x_i$ and two \textit{literal nodes} $l_i$ and $\overline l_i$ {with no demand} $b_{x_i}=b_{l_i}=b_{\overline l_i}=0$;
		\item for each clause $C_p$ one \textit{clause node} $C_p$  {with no demand} $b_{C_p}=0$.
	\end{itemize}
	
	 \noindent {The arc set $A$ is composed of the following subsets:}
	\begin{itemize}
		\item {$A_{sl}$  contains one arc from $s$ to each literal node;}
		\item {$A_{lx}$  contains one arc from each literal node to its corresponding variable node;}
		\item {$A_{lC}$  contains one arc $(l, C_p)$ for each clause $C_p$ and each literal $l$ in this clause;}
		\item {$A_{xt}$  contains one arc from each variable node to $t$;}
		\item {$A_{Ct}$  contains one arc from each clause node to $t$;}
		\item $A_{st}$ contains one arc from $s$ to $t$.
	\end{itemize}
	tea
	
	{Note that the size of the constructed graph is polynomial in $n$ and $m$. It has $3n + m + 2$ nodes and $5n + 4m+1$ arcs.}
	
	{The capacity of any arc is $m$ and its cost is $0$. Consequently, $c^*$ is equal to $0$. The nominal arc demand of all the arcs is
	$0$ except for the $n$ arcs of $A_{xt}$ where it is equal to $1$.}

	{Instance $I_{MCF}$ contains $m$ scenarios, one scenario $\xi_p$ per clause $C_p$. For each scenario $\xi_p$ the weight $w_p$ is equal to $1$, the  arc demands are all equal to $0$ except $\ell_{s,t}^{\xi_p}$ which is equal to $n-1$ and $\ell_{C_p,t}^{\xi_p}$ which is equal to $1$ (see Figure~\ref{fig:MCF_DV_P_G}). }
	
\vspace{.3cm}
\noindent{\textbf{Solving $I_{SAT}$ from a solution of $I_{MCF}$}}
		
	 {Let $f$  be a nominal feasible flow of $I_{MCF}$. Since the graph does not contain any cycle, the $n$ units supplied by $s$ are necessarily sent on $n$  different paths  to satisfy the arc demands $\ell_{x_i,t}=1$ for all $i\in\s{1, ..., n}$ (see example in
	 Figure~\ref{fig:MCF_DV_P_N}).}
		
	\begin{figure}[H]
		\begin{subfigure}{0.45\textwidth}
			\centering
			\includegraphics{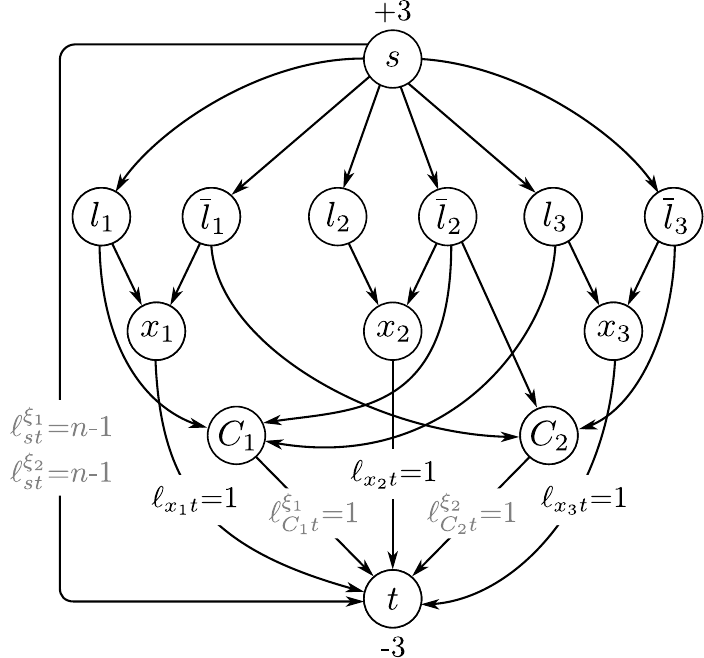}
			\caption{Graph and scenarios of $I_{MCF}$. Only the positive demands are represented.}
			\label{fig:MCF_DV_P_G} 
		\end{subfigure}
		\begin{subfigure}{0.45\textwidth}
			\centering
			\includegraphics{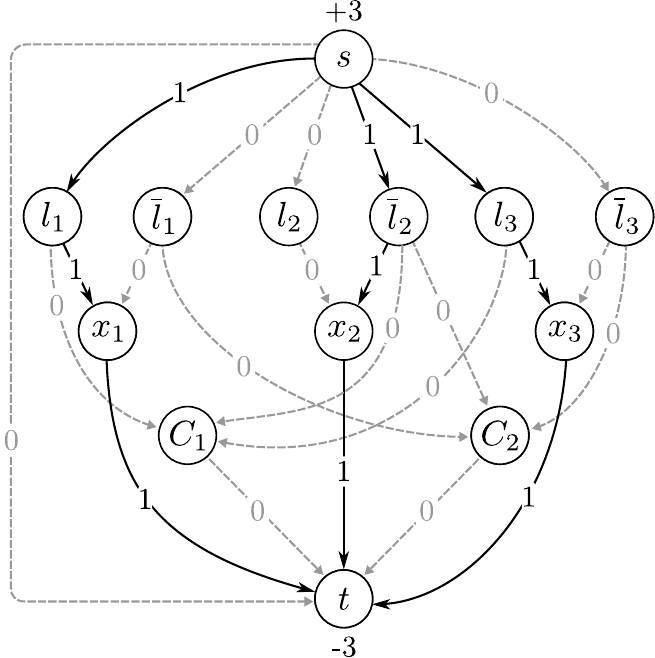}
			\caption{Nominal flow $f$.}
			\label{fig:MCF_DV_P_N} 
		\end{subfigure}
		\begin{subfigure}{0.45\textwidth}
			\centering
			\includegraphics{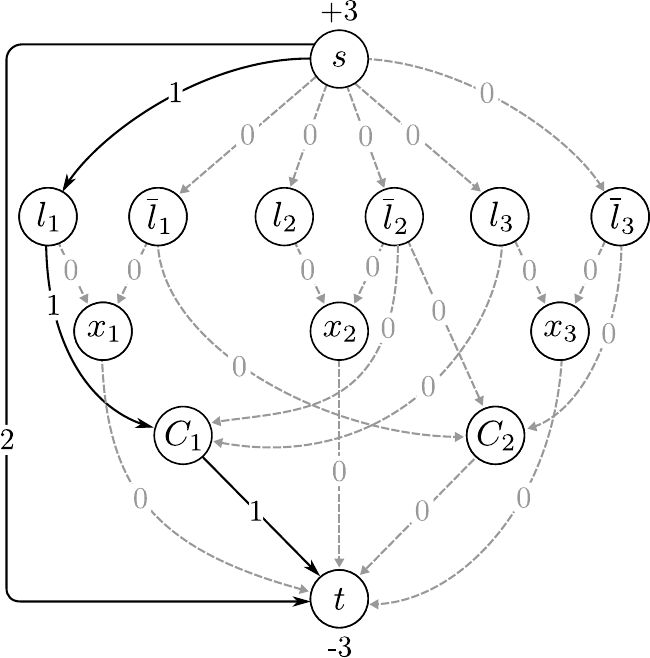}
			\caption{Flow of the first scenario $f^{\xi_1}$.}
			\label{fig:MCF_DV_P_S1} 
		\end{subfigure}\hfill
		\begin{subfigure}{0.45\textwidth}
			\centering
			\includegraphics{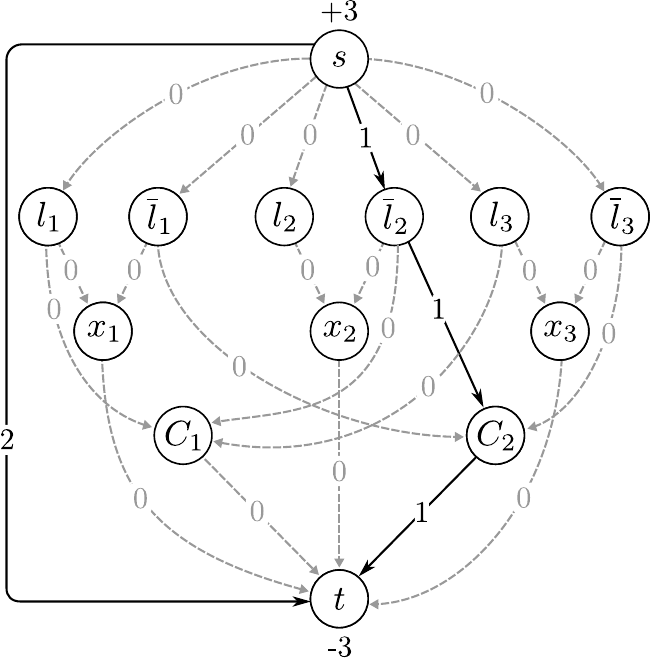}
			\caption{Flow of the second scenario $f^{\xi_2}$.}
			\label{fig:MCF_DV_P_S2} 
		\end{subfigure}
	
	  \caption{{Example of the reduction to $MCF^p_{d_{val}}$ of an instance of the $3$-SAT problem with
	  $3$ variables and $2$ clauses $C_1=(x_1\vee\overline x_2\vee x_3)$ and
	  $C_2=(\overline x_1\vee\overline x_2\vee\overline x_3)$. Figures~\ref{fig:MCF_DV_P_N},~\ref{fig:MCF_DV_P_S1} and~\ref{fig:MCF_DV_P_S2} represent an optimal solution $(f, f^{\xi_1}, f^{\xi_2})$ to $MCF^p_{d_{val}}$. The arcs with a flow of value $0$ are colored in gray.}}
	
	\label{fig:MCF_DV_P} 
	  
	\end{figure}

	{Let $f^{\xi_p}$ be the flow of scenario
	$\xi_p$. To satisfy the lower bounds $\ell_{s,t}^{\xi_p}=n-1$ and
	$\ell_{C_p,t}^{\xi_p}=1$, flow $f^{\xi_p}$ necessarily sends $n-1$ units of flow on
	 arc $(s,t)$ and $1$ unit of flow on a path $(s, l, C_p, t)$ with $l$ a
	literal included in clause $C_p$ (see Figures~\ref{fig:MCF_DV_P_S1} and~\ref{fig:MCF_DV_P_S2}). As represented in
	Table~\ref{tab:MCF_DV_P_costs}, each scenario leads to a
	solution cost of value $4n$ or $4n+2$. A cost of $4n$ is obtained if and only if the
	arc of $A_{sl}$ used in $f^{\xi_p}$ is also used in the nominal flow.}

	\begin{table}
		\centering\begin{tabular}{cccc}
			\hline		\textbf{Arc sets} & $\sum_a f_a$ & $\sum_a f^{\xi_p}_a$ & \textbf{Solution cost}
			\\\hline
	
			$A_{st}$ & $0$ & $n-1$ & $n-1$\\
			$A_{lx}$ & $n$ & $0$ & $n$\\
			$A_{xt}$ & $n$ & $0$ & $n$\\
			$A_{lC}$ & $0$ & $1$ & $1$\\
			$A_{Ct}$ & $0$ & $1$ & $1$\\
			$A_{sl}$ & $n$ & $1$ & $n-1$ or $n+1$\\\hline
			& & \textbf{Total} & $4n$ or $4n+2$\\\hline
			
		\end{tabular}
		\caption{{Solution cost incurred by scenario $\xi_p$ for each set of arcs. The solution cost associated with $A_{sl}$ is $n-1$ if the arc of $A_{sl}$ used in $f^{\xi_p}$ is also used in the nominal flow.}}
		\label{tab:MCF_DV_P_costs}
	\end{table}

	Consequently, the solution cost over all the scenarios is contained in $[4mn,
	4mn+2m]$. A cost of $4mn$ is obtained if and only if all the arcs of $A_{sl}$ used by the scenarios are also used in the nominal flow (i.e., if setting to true all the variables $x_i$  such that $f_{s,l_i} > 0$ and to false the others enables to satisfy all the clauses).
	
		\end{proof}

\begin{definition} (\cite{ahuja1988network}, Section  1.2)  A \textit{convex  cost flow  problem} is  a min-cost flow
problem where the cost of an arc is a piecewise linear convex function of its
flow.
\end{definition}

\begin{property}
$MCF_{d_{val}}^r$ is a convex cost flow problem.
\end{property}

\begin{proof}
The cost  of an arc  $a\in A$  is $|x^r_{a}-x_{a}|$ which  is a
convex piecewise linear function of  flow $x^r_{a}$.
\end{proof}

\begin{theorem}(\cite{ahuja1988network}, Section  14.3) A  convex cost
  flow  problem with  piecewise linear  convex cost  functions can  be
  transformed into a min-cost flow problem.
  \label{thm:ahuja_convex_cost_flow}
\end{theorem}

\begin{corollary}
  $MCF_{d_{val}}^r$ is a polynomial problem.
\end{corollary}

\subsection{Complexity of the robust min-cost flow with distance $d_{struct}$}
\label{sec:mcf_dstruct}

We show that  for distance $d_{struct}$ both the proactive and the reactive problems are $\NP$-hard.

\begin{theorem}$MCF_{d_{struct}}^p$ is $\NP$-hard.\label{thm:MCFadaptD}\end{theorem}

\begin{proof}

	We prove the theorem with a reduction from problem $3$-Partition. This
	problem considers a set $E=\s{1, 2, ..., 3m}$ of $3m$ elements and a positive
	integer $B$. To each element $i\in E$ is associated a size $s(i)\in]\frac B 4,
	\frac B 2[$ such that $\sum_{i\in E} s(i)=mB$. The aim of this problem is to
	determine whether there exists a partition $\s{E_1, ..., E_m}$ of $E$ into $m$
	subsets such that the size of each subset $E_j$, $\sum_{i\in E_j}s(i)$, is
	equal to $B$. Note that since $s(i)\in]\frac B 4, \frac B 2[$, {the cumulative
	sizes of $2$ and $4$ objects are respectively lower and greater than $B$. As a consequence }each
	set necessarily contains exactly $3$ elements.

	We now present how to construct an instance $I_{MCF}$ of the optimization problem 
	$MCF_{d_{struct}}^p$ from an instance $I_{3P}$ of the feasibility problem $3$-Partition and we prove that a solution of $I_{MCF}$
	has a solution cost of $7m$ if and only if $I_{3P}$ is a yes-instance. 
	
\vspace{.3cm}
\noindent{\textbf{Construction of $I_{MCF}$}}

	{Figure~\ref{fig:redGraph} illustrates the graph obtained for an
 	instance of problem $3$-Partition  with $m=2$ and $B=100$.} 
	
	The set of nodes
	$V$ of the constructed graph $G=(V, A)$ is composed of one node $\alpha$, one
	node $V_i$ for each element $i\in E$ and one \textit{subset node}
	$E_j$ for each subset $j\in\s{1, ..., m}$. The demand of all the nodes is equal to $0$.
	
	\begin{figure}[H]
	
	\begin{subfigure}[t]{0.46\textwidth}
	  \centering
	  \includegraphics{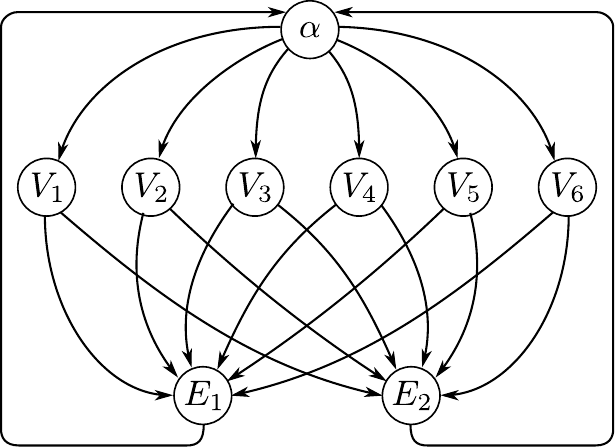}
	  \caption{{Graph of $I_{MCF}$.     All node demands are $0$. For any arc, demand is $0$, capacity is $B$ and unitary cost is $1$.}}
	\label{fig:redGraph} 
	\end{subfigure}\hspace{.5cm}
	\begin{subfigure}[t]{0.46\textwidth}
	  \centering
	  \includegraphics{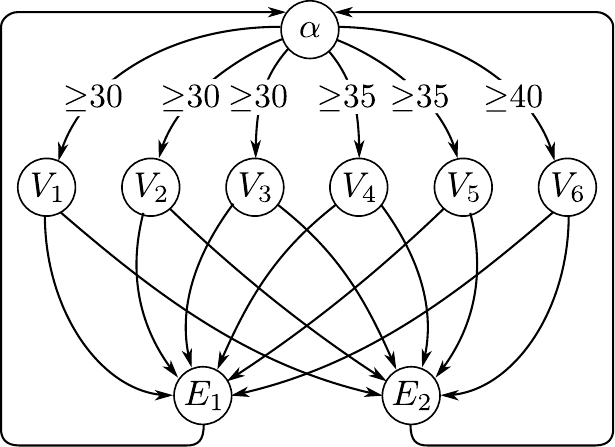}
	  \caption{{Positive demands of the scenario.}}
	\label{fig:redDemands} 
	\end{subfigure}
	\centering\begin{subfigure}[t]{0.46\textwidth}
	  \centering
	  \includegraphics{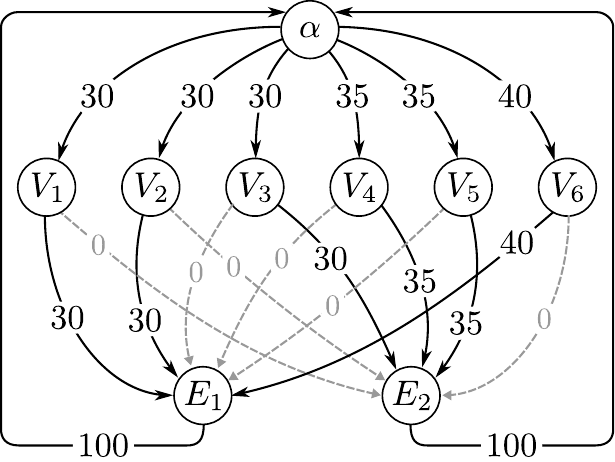}
	  \caption{{Optimal proactive flow inducing the $3$-partition $E_1=\s{1, 2, 6}$ and $E_2=\s{3, 4, 5}$. The arcs with a flow of value $0$ are colored in gray.}}
	\label{fig:redsolution} 
	\end{subfigure}
	\caption{{Example of the reduction to $MCF^p_{d_{struct}}$ of an instance of the $3$-Partition problem with   $m=2$, $B=100$ and elements of size $(30, 30, 30, 35, 35, 40)$.}}
	\label{fig:redDR}
	\end{figure}
	
	The arc set $A$ contains one arc $(\alpha, V_i)$ for each $i\in E$ and one arc
	$(V_i, E_j)$ for each $i\in E$ and each $j\in\s{1, ..., m}$. The graph also contains an additional set of arcs, denoted by $A_{E\alpha}$, which contains one arc from
	each subset node to $\alpha$. Note that the size of the graph is polynomial in $n$ and $m$ as it
	contains $4 m + 1$ nodes and $3m^2+4m$ arcs.

    For any arc, the nominal demand is $0$, the capacity is $B$ and the unitary cost is $1$. It follows, that the only optimal nominal solution is the empty flow, with the objective value $c^*=0$.
    
    One unique scenario $\xi$ is considered in which the only arc demands different from $0$ are  $\ell^\xi_{\alpha, V_i}$ of arcs $(\alpha, V_i)$ for all $i\in E$ which are set to $s(i)$ {(see Figure~\ref{fig:redDemands})}.

	 	 \vspace{.3cm}
\noindent{\textbf{Solving $I_{3P}$ from a solution of $I_{MCF}$}}

	 Since $c^*=0$, the only feasible proactive solution is the empty flow. 

	Let $f^\xi$ be a feasible solution for  scenario $\xi$. Observe that the outgoing flow from $\alpha$ is at least $Bm$. Moreover, arcs $(E_j, \alpha)$ have a capacity of $B$. Thus, the arcs $(E_j,\alpha)$ must all be saturated and the flow running through $\alpha$ is exactly $Bm$. Since the flow on each arc $(E_j, \alpha)$ is $0$ in the nominal flow and $B$ in $f^\xi$, these $m$ arcs induce a solution cost of $m$.
	

	We now show that at least an additional cost of $2|E|$ is incurred by the remaining arcs. Each $V_i$ has an ingoing flow of value $s(i)$. Consequently, $f^\xi$ must send a total flow of $s(i)$ on at least one of the $m$ following paths $\s{(\alpha, V_i, E_j)}_{j=1}^m$. To sum up, $f^\xi$ uses at least $2$ additional arcs for each element $i\in E$, leading to a solution cost of $m+2|E|$
	(see Figure~\ref{fig:redsolution}).  The use of several paths to satisfy a demand $\ell^\xi_{\alpha V_i}$ would increase the solution cost as several
	arcs from $\s{(V_i, E_j)}_{j=1}^m$ would be used.


	Consequently, a solution cost of $7m$ is obtained if and only if for each $i\in E$ exactly one path is used to satisfy the demand $\ell_{\alpha, V_i}$   (i.e., if a partition with subsets of size $B$ is obtained by assigning each element $i$ to the unique set $E_j$ such that $f^\xi_{V_i, E_j}> 0$).
	
	\end{proof}

\begin{theorem}$MCF_{d_{struct}}^r$ is $\NP$-hard.
\end{theorem}

\begin{proof} 

{In the reduction considered in the proof of Theorem~\ref{thm:MCFadaptD}, the proactive flow is unique and known a priori as it is necessarily empty. Consequently, this flow can be given as an input of the problem rather than being part of its solution without altering the validity of the reduction.}

{This new reduction leads to an instance of $MCF_{d_{struct}}^r$ since the nominal flow is fixed and only $1$ scenario is considered.}
	\end{proof}

\section{Solution robustness for max-flow problem with uncertain arc capacities}
\label{sec:mf}

The max-flow problem can be stated as:

\vspace{.3cm}

\noindent\fbox{
	\parbox{0.97\textwidth}{\textsc{Max-flow Problem} \newline
		\begin{tabular}[\textwidth]{p{0.06\textwidth}p{0.85\textwidth}}

Input: & A digraph $G = (V, A)$ with a source $s\in V$, a sink $t\in V$ and capacities $u_a\in\mathbb N$ on the flow of each arc $a\in A$.\\ 

Output: & Find an integer flow with maximum value.
		\end{tabular}
	}
}
\vspace{.3cm}

We consider max-flow problems with uncertainties on the capacities $u$. Let
$\mathcal S$ be a set of scenarios and note $u^\xi$ the capacities associated
to a scenario $\xi\in \mathcal S$. {Note that the associated reactive and proactive problems are not particular cases of the ones considered in the previous section as the uncertainty is not on the arc demands.}

In the remaining of this section we determine the complexity of the
four reactive and proactive max-flow problems associated with distances $d_{val}$ and $d_{struct}$, as represented in
Table~\ref{tab:complexityMF}.

\begin{table}[H]
    \centering
    \renewcommand{\arraystretch}{1.4}

	\begin{tabular}{ccc}
		\hline
			 \multirow{2}{*}{\textbf{Problem}}& \multicolumn{2}{c}{\textbf{Distance}}\\
			 & $d_{val}$ & $d_{struct}$ \\\hline
			 \multirow{2}{*}{\textbf{Proactive}} & \multirow{2}{*}{$\NP$-hard (Section~\ref{sec:mf_dval})} & $\NP$-hard (Section~\ref{sec:mf_dstruct})\\
			 & & \small{even with $1$ scenario}\\\hline
			 {\textbf{Reactive}} & {Polynomial (Section~\ref{sec:mf_dval})} & $\NP$-hard (Section~\ref{sec:mf_dstruct})\\\hline
		\end{tabular}
    \caption{{Complexity of the max-flow  problems with uncertain arc capacities associated with distances $d_{val}$ and $d_{struct}$.}}
    \label{tab:complexityMF}
\end{table}

\subsection{Complexity of the robust max-flow with distance $d_{val}$}
\label{sec:mf_dval}

We show in this section that  for distance $d_{val}$ the proactive problem is $\NP$-hard and that the reactive problem is polynomial.

\begin{theorem}
	$MF_{d_{val}}^p$ is $\NP$-hard.

  \end{theorem}

\begin{proof}
The proof of this theorem is similar to the one of Theorem~\ref{th:MCF_DV_P}. An instance $I_{MF}$ of the optimization problem $MF_{d_{val}}^p$ is constructed from an instance $I_{SAT}$ of the feasibility problem $3$-SAT and we prove that $I_{MF}$ has an optimal solution with  a solution cost equal to $m(3n+2m-1)$ if and only if $I_{SAT}$ is a yes-instance. 

\vspace{.3cm}
\noindent{\textbf{Construction of $I_{MF}$}}

As represented in Figure~\ref{fig:MF_DV_P}, the graph associated with $I_{MF}$ has the same nodes and arcs than the one in the proof of Theorem~\ref{th:MCF_DV_P} except for the arc $(s,t)$ which does not exist. The graph also includes an additional set of arcs $A_{sC}$ which contains one arc between $s$ and each clause node.

{
The  nominal capacities of all the arcs is {$+\infty$ except for the arcs in $A_{xt}$ and $A_{Ct}$ for which it is equal to $1$ and the arcs in $A_{lC}$ for which it is null.}} {As proved below, $c^*$ is equal to $n+m$.}
	
\begin{figure}
	\begin{subfigure}{0.45\textwidth}
		\centering
		\includegraphics{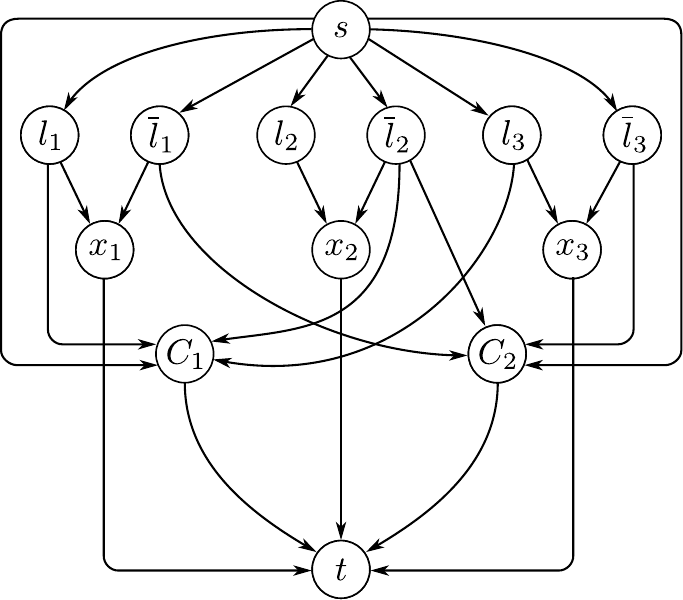}
		\caption{Graph of $I_{MF}$. }
		\label{fig:MF_DV_P_G} 
	\end{subfigure}
	\begin{subfigure}{0.45\textwidth}
		\centering
            \includegraphics{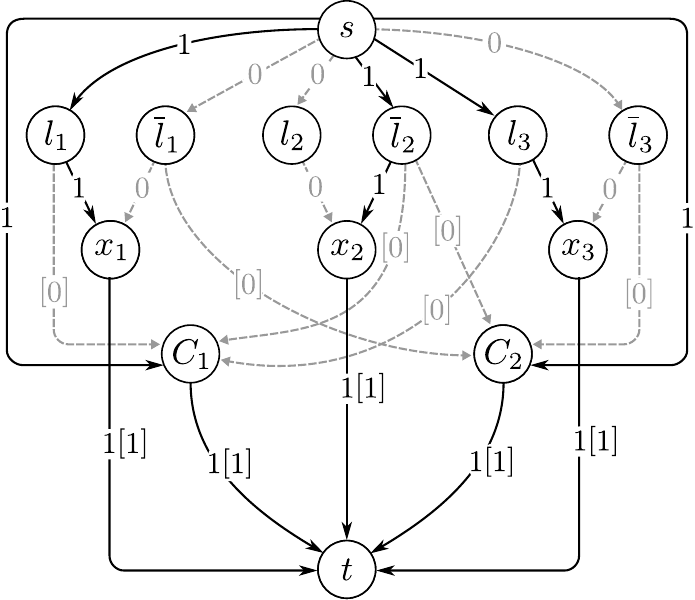}
		\caption{Nominal flow $f$. {The nominal capacities different from $+\infty$ are represented between brackets.}}
		\label{fig:MF_DV_P_N} 
	\end{subfigure}
	\begin{subfigure}{0.45\textwidth}
		\centering
		\includegraphics{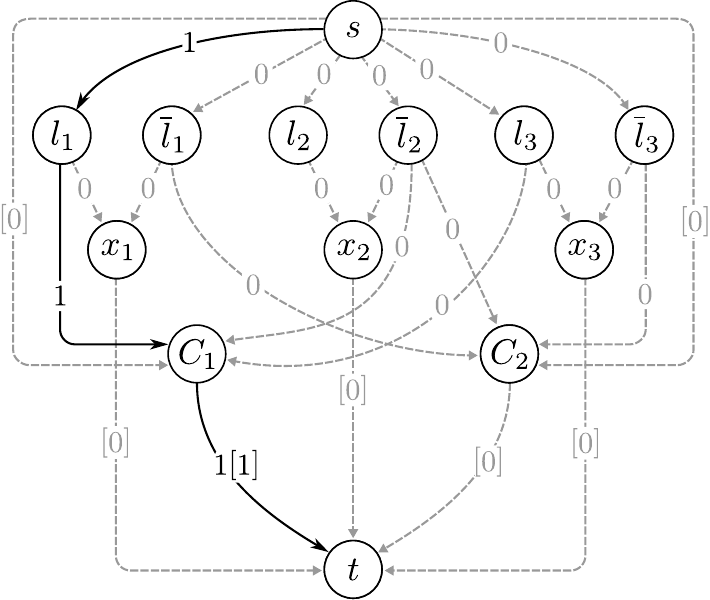}
		\caption{Flow $f^{\xi_1}$ of the first scenario $\xi_1$. {The scenario capacities different from $+\infty$ are represented between brackets.}}
		\label{fig:MF_DV_P_S1} 
	\end{subfigure}\hfill
	\begin{subfigure}{0.45\textwidth}
		\centering
		\includegraphics{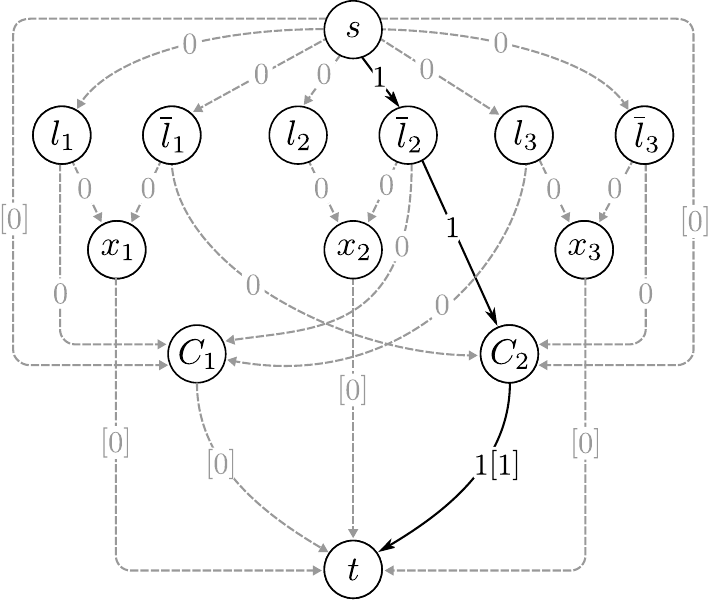}
		\caption{Flow $f^{\xi_2}$ of the second scenario $\xi_2$. {The scenario capacities different from $+\infty$ are represented between brackets.}}
		\label{fig:MF_DV_P_S2} 
	\end{subfigure}

  \caption{{Example of the reduction to $MF^p_{d_{val}}$ of an instance of the $3$-SAT problem with
  $3$ variables and $2$ clauses $C_1=(x_1\vee\overline x_2\vee x_3)$ and
  $C_2=(\overline x_1\vee\overline x_2\vee\overline x_3)$ to $MF^p_{d_{val}}$. Figures~\ref{fig:MF_DV_P_N},~\ref{fig:MF_DV_P_S1} and~\ref{fig:MF_DV_P_S2} represent an optimal solution $(f, f^{\xi_1}, f^{\xi_2})$. The arcs with a flow of value $0$ are colored in gray.}}

\label{fig:MF_DV_P} 
  
\end{figure}
	
	{Instance $I_{MF}$ contains $m$ scenarios $\s{\xi_1, ..., \xi_m}$ with weights equal to $1$. The capacities of the arcs of a scenario $\xi_p$ are all equal to {$+\infty$ except  the arcs in $A_{xt}$, $A_{sC}$ and $(C_q, t)$ for all $q\neq p$ for which it is equal to $0$ and $u_{C_p, t}$ which is equal to $1$}  (see Figure~\ref{fig:MF_DV_P_G}). }

\vspace{.3cm}
\noindent{\textbf{Solving $I_{SAT}$ from a solution of $I_{MF}$}}

Let $f$  be a nominal feasible flow of $I_{MF}$. Since the $n+m$ arcs arriving in $t$ all have a nominal capacity of $1$, the  nominal flow value is at most $n+m=c^*$.	

 Flow $f$ necessarily sends one unit of flow for each clause $C_p$ on path $(s, C_p, t)$ and one unit of flow for each variable $x_i$ on one of the following two paths $(s, l_i, x_i, t)$ or $(s, \overline l_i, x_i, t)$ (see example in
 Figure~\ref{fig:MF_DV_P_N}).

{Let $f^{\xi_p}$ be the flow of scenario
$\xi_p$. Due to the scenario capacities, $f^{\xi_p}$ can either be an empty flow (solution $f^{\xi_p,0}$) or send $1$ unit of flow from $s$ to $t$ through a path $(s, l, C_p, t)$ with $l$ a literal of clause $C_p$ (solution $f^{\xi_p,l}$) (see Figures~\ref{fig:MF_DV_P_S1} and~\ref{fig:MF_DV_P_S2}). As represented in
Table~\ref{tab:MF_DV_P_costs}, solution $f^{\xi_p,0}$ leads to a
solution cost of value $3n+2m$ and $f^{\xi_p,l}$ to a solution cost in $[3n+2m-1,3n+2m+1]$. A cost of $3n+2m-1$ is obtained if and only if the
arc of $A_{sl}$ used in $f^{\xi_p,l}$ is also used in the nominal flow.}

\begin{table}
	\centering\begin{tabular}{cccccc}
		\hline		\multirow{2}{*}{\textbf{Arc sets}} & \multirow{2}{*}{$\sum_a f_a$} & \multirow{2}{*}{$\sum_a f^{\xi_p,0}_a$} & \multirow{2}{*}{$\sum_a f^{\xi_p,l}_a$} & \multicolumn{2}{c}{\textbf{Solution cost}}\\    

		& & & & $\sum_a |f^{\xi_p,0}_a-f_a|$ & $\sum_a |f^{\xi_p,l}_a - f_a|$
		\\\hline

		$A_{sC}$ & $m$ & $0$ & $0$ & $m$ & $m$\\
		$A_{lC}$ & $0$ & $0$ & $1$ & $0$ & $1$\\
		$A_{Ct}$ & $m$ & $0$ & $1$ & $m$ & $m-1$\\
		$A_{sl}$ & $n$ & $0$ & $1$ & $n$ & $n\pm 1$ \\
		$A_{lx}$ & $n$ & $0$ & $0$ & $n$ & $n$\\
		$A_{xt}$ & $n$ & $0$ & $0$ & $n$ & $n$\\\hline
		& & & {\textbf{Total}} & {$3n+2m$} & $3n+2m\pm 1$\\\hline
		
	\end{tabular}
	\caption{{Solution cost incurred by scenario $\xi_p$ for each set of arcs.  The solution cost is reduced  for $A_{sl}$ if the arc used in   $f^{\xi_p,l}$ is also used in the nominal flow.}}
	\label{tab:MF_DV_P_costs}
\end{table}

Consequently, the solution cost over all the scenarios is between $m (3n+2m-1)$ and $m (3n+2m+1)$. An optimal cost of $m(3n+2m-1)$ is obtained if and only if all the arcs of $A_{sl}$ used by the scenarios are also used in the nominal flow (i.e., if setting to true all the variables $x_i$  such that $f_{s,l_i} > 0$ and to false the others enables to satisfy all the clauses).

	\end{proof}

	We show that $MF_{d_{val}}^r$ corresponds to a \textit{convex cost flow problem} which is known to be polynomial.
	
	\begin{property}
	$MF_{d_{val}}^r$ is a polynomial problem.
	\end{property}
	
	\begin{proof}
		$MF_{d_{val}}^r$ is a convex cost flow problem since the cost  of an arc  $a\in A$  is $|x^r_{a}-x_{a}|$. {Since a max-flow problem is a particular min-cost flow problem, according to Theorem~\ref{thm:ahuja_convex_cost_flow}, $MF_{d_{val}}^r$ is a polynomial problem.}
	\end{proof}

\subsection{Complexity of the robust max-flow with distance $d_{struct}$}
\label{sec:mf_dstruct}

We show that   for distance $d_{struct}$ both the proactive and the reactive problems are $\NP$-hard.

\begin{theorem}$MF_{d_{struct}}^p$ is $\NP$-hard.
\label{thm:MFadaptD}\end{theorem}
\begin{proof}

	Similarly to the proof of $MCF_{d_{val}}^p$, we use a reduction
	from problem $3$-SAT. 
	
 Let $I_{SAT}$ be an instance of the feasibility problem $3$-SAT with $m$ clauses and $n$
 variables. As represented in Figure~\ref{fig:MF_DS_P}, we create an instance
 $I_{MF}$ of the optimization problem $MF_{d_{struct}}^p$ with one unique scenario $\xi$. We prove that an
 optimal solution of $I_{MF}$ leads to a solution cost of $3n + 2m$ if and only
 if $I_{SAT}$ is a yes-instance.
	
\vspace{.3cm}
\noindent{\textbf{Construction of $I_{MF}$}}
	
	In $I_{MF}$, the set of nodes $V$  of graph
	$G=(V, A)$ is composed of: 
	
	\begin{itemize} 
		
		\item $2$ nodes $s$ and $t$; 
		
		\item
	{$2n$ \textit{literal nodes} $\s{l_i}_{i=1}^{n}$ and $\s{\overline
	l_i}_{i=1}^{n}$; }
	
	\item {$3n$
	\textit{variable nodes} $\s{x^1_i, x^2_i,x^3_i}_{i=1}^n$; }
	
	\item {$2m$
	\textit{clause nodes} $\s{C^1_p,C^2_p}_{p=1}^{m}$.}
	
	\end{itemize}
	
	\noindent {The arc set $A$ is composed of the following subsets:}
	\begin{itemize}
		
  \item {$A_{s,l}$ which contains one arc from $s$ to each literal node;}

  \item {$A_{s,x}$ which contains one arc from $s$ to each first variable node $x^1_i$;}
  
  \item {$A_{l,x}$ which contains one arc from each literal node to its corresponding first variable node
  $x^1$;}

  \item {$A_{x}$ which contains arcs $(x^1_i, x_i^2)$ and $(x^2_i, x^3_i)$ for each $i$;}

  \item {$A_{x, t}$ which contains one arc from each last variable node $x^3_i$ to $t$;}

  \item {$A_{s, C}$ which contains one arc from $s$ to each first clause node
  $C^1_p$;}

  \item {$A_{l,C}$ which contains one arc $(l, C^1_p)$ for each clause $C_p$
  and each literal $l$ in $C_p$. Hence, $3$ such arcs exist for each clause;}

  \item {$A_C$ which contains one arc  $(C^1_p, C^{2}_p)$ for each clause $C_p$;}

		\item {$A_{C, t}$ which contains one arc from each  clause node $C^{2}_p$ to $t$.}
		
	\end{itemize}
	
 {Note that the size of the constructed graph is polynomial in $n$ and $m$. It has $5n+2m+2$ nodes and $8n+6m$ arcs.}

 {As proved below, the value of $c^*$ is equal  to $n+m$. The
 nominal capacities $u$ and the scenario capacities $u^\xi$ for each  arc subset are presented in
 Table~\ref{tab:MF_DS_P}.}

	\begin{figure}
		\begin{subfigure}{0.45\textwidth}
			\centering
			\includegraphics{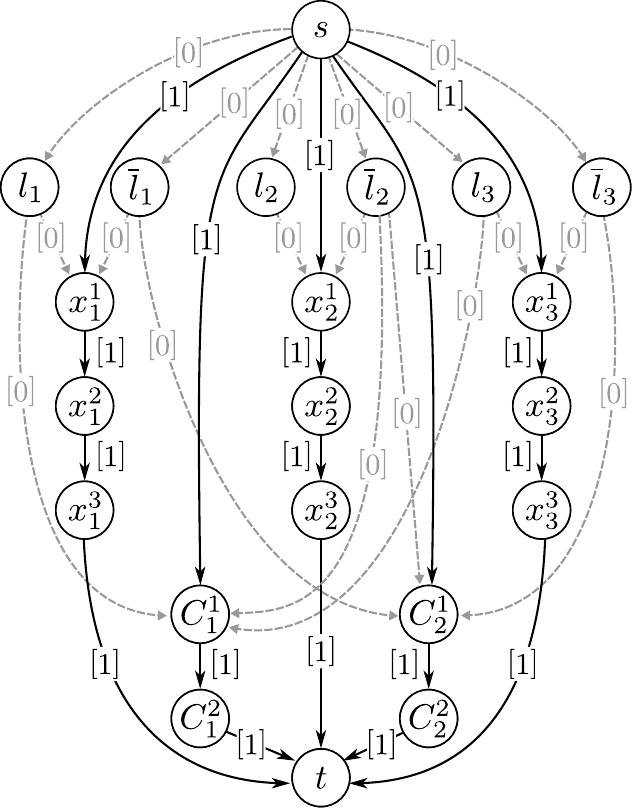}
			\caption{Nominal capacities.} 
		\end{subfigure}
		\begin{subfigure}{0.45\textwidth}
			\centering
			\includegraphics{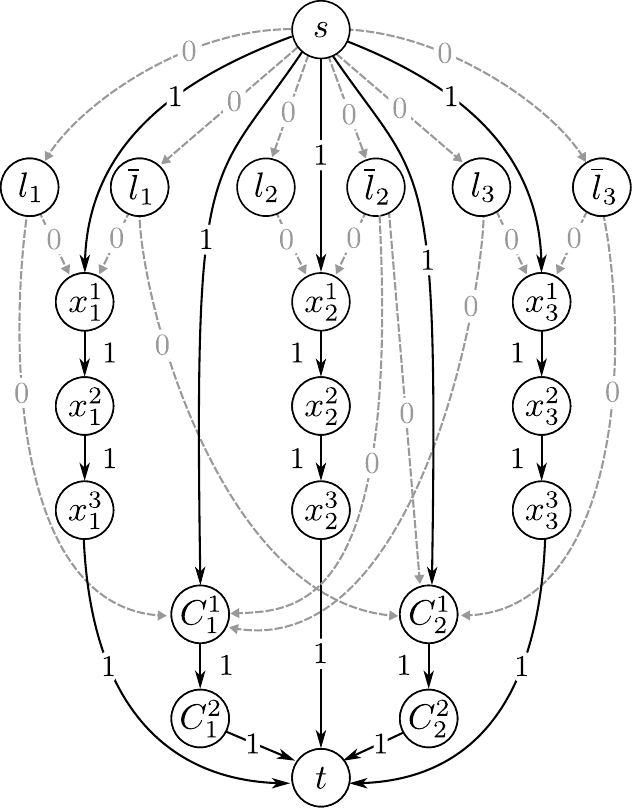}
			\caption{Unique optimal nominal solution with flow value $n+m$.}
			\label{fig:MF_DS_P_N} 
		\end{subfigure}
		\begin{subfigure}{0.45\textwidth}
			\centering
			\includegraphics{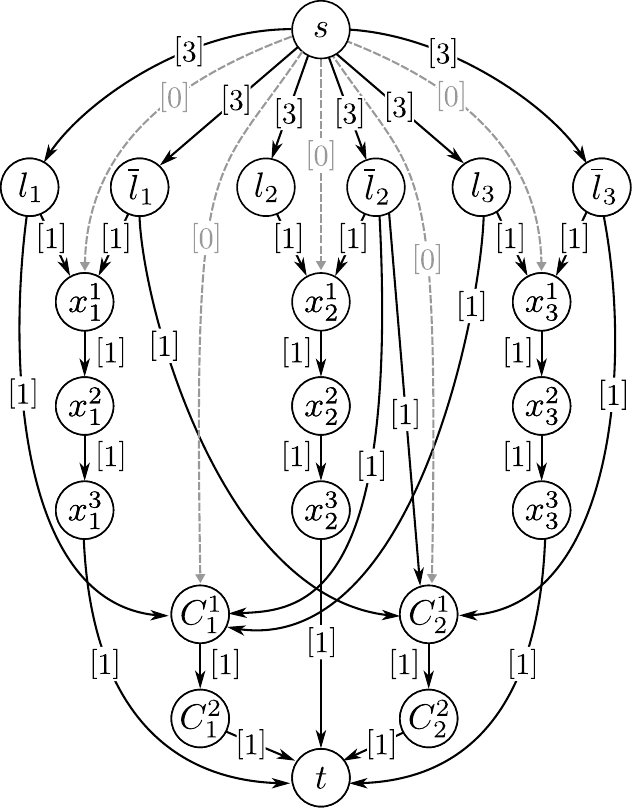}
			\caption{Capacities of the scenario.}
			\label{fig:MF_DS_P_RComplet} 
		\end{subfigure}\hfill
		\begin{subfigure}{0.45\textwidth}
			\centering
			\includegraphics{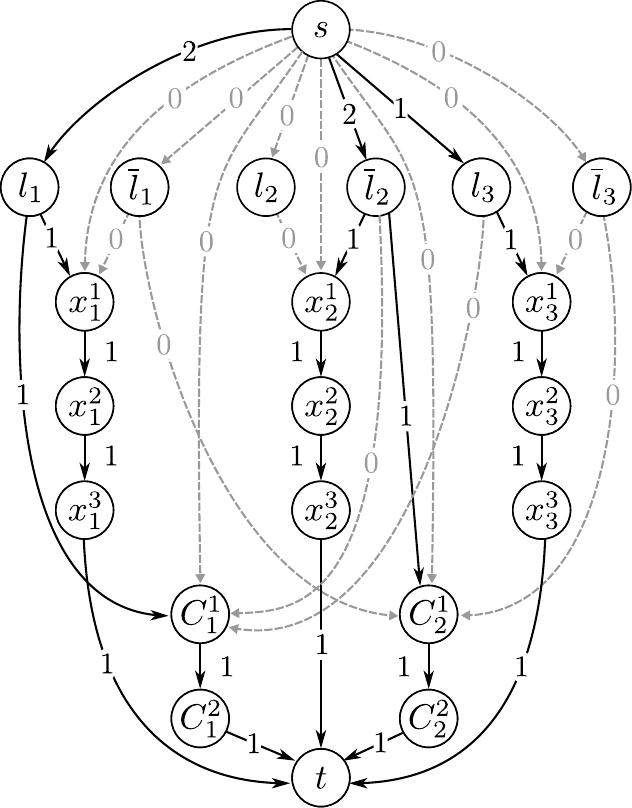}
			\caption{Flow of the scenario of an optimal solution.}
			\label{fig:MF_DS_P_RSolution} 
		\end{subfigure}
	
	  \caption{{Example of the reduction to $MF^p_{d_{struct}}$ of an instance of the $3$-SAT problem with
	  $3$ variables and $2$ clauses $C_1=(x_1\vee\overline x_2\vee x_3)$ and
	  $C_2=(\overline x_1\vee\overline x_2\vee\overline x_3)$. The arcs with a capacity or a flow of value $0$ are colored in gray.}}
	
	\label{fig:MF_DS_P} 
	  
	\end{figure}
	
	\begin{table}
		\centering\begin{tabular}{lcc}
			\hline
			\multirow{2}{*}{\textbf{Arc sets}} & \multicolumn{2}{c}{\textbf{Capacities}} \\
				& $u$ & $u^\xi$\\
	
			\hline
			$A_{s,l}$  & $0$ & $m+1$ \\

			$A_{s,x}$, $A_{s, C}$ & $1$ & $0$\\

			$A_{l,x}$, $A_{l,C}$ & $0$ & $1$\\

			$A_{x}$, $A_{x,t}$, $A_C$, $A_{C,t}$ & $1$ & $1$\\

			 \hline
	
		\end{tabular}
		\caption{Capacities of each set of arcs for the nominal case and scenario $\xi$ in instance $I_{MF}$.}
		\label{tab:MF_DS_P}
	\end{table}

\vspace{.3cm}
\noindent{\textbf{Solving $I_{SAT}$ from a solution of $I_{MF}$}}

{Let $f$  be a nominal flow of $I_{MF}$. As examplified in
Figure~\ref{fig:MF_DS_P_N}, the maximal nominal flow is $c^*=n+m$ and it can only be reached by a unique flow which sends:}
   
   \begin{itemize}
   
   \item {$1$ unit for each clause $C_p$
   through path $(s, C^1_p, C^2_p,t)$; and}
   
   \item {$1$ unit for each variable  $x_i$
	through path $(s,x^1_i, x^2_i, x^3_i,t)$.}
   
   \end{itemize}

{Let $f^\xi$ be the flow of the
scenario. An empty flow $f^\xi$ would leads to a solution cost of $4n+3m$ (i.e., the
number of arcs used by the nominal flow). This can be improved by sending, for
each variable $x$, $1$ unit of flow along one of the two following paths:
$P_x=(s,l, x^1, x^2, x^3, t)$ or $\overline P_x=(s, \overline l, x^1, x^2, x^3, t)$.
Let assume without loss of generality
 that only $P_x$ is used. This leads to a reduction of $1$ of the solution
cost. Indeed, $3$ arcs of the path are also used in the nominal flow ($(x^1, x^2)$,
$(x^2, x^3)$, $(x^3, t)$) and $2$ are not ($(s, l)$ and $(l, x^1)$). 
Note that both $P_x$ and $\overline P_x$ cannot be used simultaneously as the capacity of $(x^1, x^2)$ is equal to $1$.
Thus, the solution cost of $f^\xi$ is reduced by $n$ by sending $1$ unit of flow on either $P_x$ or $\overline P_x$ for each variable $x$ (see
Figure~\ref{fig:MF_DS_P_RSolution}). Let $A^\xi_{sl}$ be the subset of arcs of  $A_{sl}$ now
used in $f^\xi$.}

 {To further reduce the solution cost, $f^\xi$ must use arcs from $A_C$ and
 $A_{C,t}$. The clause nodes associated with a clause $C$ can only be reached
 by $f^\xi$ through a path $(s, l, C^1, C^{2}, t)$ with $l$ one of the
 three literals included in $C$. If $(s,l)\in A^\xi_{sl}$ the use of this path
 reduces by $1$ the solution cost (since $(C^1, C^2)$ and $(C^2,t)$ are also used in
 $f$ and $(l, C^1)$ is not), otherwise the solution cost remains the same. Thus,
 the solution cost can be further reduced by at most  $m$ if each clause node is reached
 through a path which includes an arc in $A^\xi_{sl}$.}

{In conclusion, there exists an optimal solution to $I_{MF}$ with  solution cost equal to $3n+2m$ if and only if $I_{SAT}$ is satisfiable. In that case, the solution of $I_{SAT}$ consists in fixing to true the variables $x_i$ such that $f^\xi_{s,l_i}>0$ and to false the others. }


	\end{proof}
	
\begin{corollary}$MF_{d_{struct}}^r$ is $\NP$-hard.
\end{corollary}

\begin{proof} 

{The reduction used in the proof of Theorem~\ref{thm:MFadaptD} can also be used
to prove this theorem since it only requires one scenario. The main difference
is that the only feasible nominal solution (examplified in
Figure~\ref{fig:MF_DS_P_N}), is given as an input of the reactive problem.}

\end{proof}

  \renewcommand{\arraystretch}{1.3}
\section{A case study - The Line Optimization Problem with uncertain demands}
\label{sec:experiments}

{In this section, we apply our solution robustness approach to a railroad planning problem inspired from~\cite{bussieck1997optimal,goossens2004branch}. We compare our approach with two approaches from the literature and highlight their similarities and differences. We then assess the solution cost reduction when relaxing the optimality constraint on the nominal objective of our approach. Finally, we show the efficiency of the proactive approach over the reactive approach.}

The Line Optimization Problem occurs in a railway system with periodic timetables. A line in an urban transportation network is a path from a departure  station to an arrival station  with stops in intermediary stations. The frequency of a line is the number of times a train has to be operated on this line in a given time interval in order to cover passenger demands. We consider the problem of determining the lines deployment and frequencies in order to minimize  the line costs.

\subsection{Deterministic problem description}
\label{sec:det_ol}
We are given the set $\SetStations$ of stations and the network $\mathcal N=(\SetStations, E,d)$ where edges of $E$ represent the undirected links between stations and $d_e$ is the length of $e\in E$.
 An origin-destination matrix OD provides the number of passengers that need to travel from any station $s$ to any other station $s'$, within one hour.

We consider a set  $L$ of potential lines. A  line $\ell$  is characterised  by a  pair $(s^1_\ell,  s^2_\ell)$ of
departure  and  arrival  stations  together with  a shortest  path linking the two stations.  Trains are planned to run on the deployed lines and all trains are supposed to have the same carriage capacity of $C$ passengers. Therefore, if a line $\ell$ is scheduled  with frequency $f$, then up to $C \times f$ passengers per hour may be carried on each edge of line $\ell$. Due to technical considerations, the frequency of any line is limited by a given value   $\mbox{MF}\in\mathbb N$. We are also given the cost $K_\ell$ of deploying line $\ell$ and the unitary cost $K'_\ell$ associated to frequencies on line $\ell$.

We introduce a binary variable $x_\ell$  equal to $1$ if and only if
line $\ell\in L$ is
deployed  and  $0$  otherwise.  When  $\ell$  is  deployed,  variable
$f_\ell$ indicates  the hourly  frequency of  line $\ell$.  A solution
$(x, f)$ is feasible for a given origin-destination matrix OD if it is included in set $\mathcal F(\mbox{OD})$ defined by the following constraints:

\begin{numcases}{\mathcal  F(\mbox{OD})}
 C  \times  \displaystyle\sum_{\ell\in  L:   e  \in  \ell}  f_{\ell}  \geq
      \displaystyle\sum_{\ell\in L: e \in \ell} \mbox{OD}(s^1_\ell, s^2_\ell) & $e \in E$ \label{eq:demand}\\
f_\ell \leq \mbox{MF}\,x_\ell & $\ell\in L.$ \label{eq:link}\\
x_\ell \leq f_\ell & $\ell\in L$ \label{eq:link2}\\
f_\ell  \in \mathbb{Z}_+ & $\ell\in L$ \nonumber\\
x_\ell \in \{0,1\} & $\ell\in L$ \nonumber
\end{numcases} 

  \noindent Constraints~\eqref{eq:demand} force the demand satisfaction of each edge. Constraints~\eqref{eq:link} and~\eqref{eq:link2} link variables $x_\ell$ and $f_\ell$. If line $\ell$ is not deployed then $x_\ell$ and $f_\ell$ are equal to zero, otherwise $f_\ell$ is not larger than $\mbox{MF}$.

\bigskip

The nominal objective of a solution $(x, f)$ is defined as

\begin{equation}
  \label{eq:deterministicObjective}
  \mathcal NO(x, f) = \sum_{\ell\in L}(K_\ell x_\ell + K'_\ell f_\ell)
\end{equation}

Let OD$^0$ be the nominal origin-destination matrix. The deterministic model
is defined as

\begin{center}
  $\overline P\left\{
    \begin{array}{ll}
 \min ~  \mathcal NO(x^0, f^0)\\
\mbox{s.t.}  ~(x^0, f^0)\in \mathcal F(OD^0)
    \end{array}
\right.$
\end{center}

We now assume that the OD-matrix is uncertain within a set $\mathcal S$ of scenarios. We are given the  set
$\s{OD^1, ..., OD^{|\mathcal S|}}$ of alternative matrices. {To handle the uncertainty, this case study compares three robust approaches:  our proactive approach, the anchored approach~\cite{bendotti2019anchor} and the $k$-distance approach~\cite{busing2012recoverable}. Each approach either constrains or optimizes differences between the nominal solution and the solution of the scenarios. Our proactive approach minimizes an average distance between the nominal and the scenario solutions. The anchored approach maximises the number of variables which value is identical in all solutions while the $k$-distance approach requires that the value of at most $k$ variables can differ in the nominal solution and the solution of each scenario.} 

{We consider two cases in which the differences between solutions are either measured on the frequency variables $f$ or on the  line deployment variables $x$. For the proactive approach this can be viewed as considering the distance in values $d_{val}$ over $f$ in the first case and  the distance in structure $d_{struct}$ over the lines deployment variables $x$ in the second case.}

{We show  how this problem can be modelled for each approach as a compact integer linear program. The compactness of the formulations is made possible by the fact that the uncertainty is represented by a finite set of scenarios $\mathcal S$ rather than by an uncertainty set.}

\subsection{Solution cost based on frequencies}
\label{sec:robfreq}

{In this section, the differences between solutions are measured through the frequency variables $f^0$ of the nominal solution and $\s{f^s}_{s\in\mathcal S}$ of the scenario solutions.}
\medskip

\noindent \textbf{Proactive approach}\\

As presented  previously,  the proactive approach minimizes the {average solution cost between  the
nominal solution $(x^0, f^0)$ and the scenario solutions $\s{(x^s, f^s)}_{s\in \mathcal S}$} while {setting the nominal objective $\mathcal NO(x^0, f^0)$  to its optimal value $c^*$}. {The  solution cost based on frequencies corresponds to the distance in values $d_{val}$ from Definition~\ref{def:value} applied to the frequency variables $f^0$ and $\s{f^s}_{s\in\mathcal S}$ (i.e., $\sum_{s\in \mathcal S}\sum_{\ell\in L} |f^0_{\ell}-f^s_{\ell}|$)}. We consider variable $df_\ell^s$  equal to
$|f^s_\ell-f^0_\ell|$ for each $\ell\in L$ and $s \in\mathcal S$:

\begin{center}
  $P^p_f(\mathcal S, c^*)\left\{
    \begin{array}{lll}
 \min & \sum_{s\in\mathcal S}\sum_{\ell\in L} df^s_\ell\\
      \mbox{s.t.} & (x^0, f^0)\in \mathcal F(OD^0)\\
&      (x^s, f^s)\in \mathcal F(OD^s) & s\in\mathcal S\\
  &     \mathcal NO(x^0, f^0)= c^*\\
 &     f^s_\ell - f^0_\ell \leq df^s_\ell & s\in\mathcal S,~\ell\in L\\
 &     -f^s_\ell + f^0_\ell \leq df^s_\ell & s\in\mathcal S,~\ell\in L
    \end{array}
\right.$
\end{center}

\bigskip

\noindent \textbf{Anchored approach}\\

A nominal variable is said to be \textit{anchored} {if its value in the nominal solution and in all scenario solutions are equal}.  For each
line $\ell\in L$,  variable  $a_\ell$ is equal to $1$ if  $f^0_\ell$ is anchored and $0$ otherwise. This
approach  maximizes the  number of  anchored variables  while ensuring
that the nominal objective of the nominal solution is optimal:

\begin{center}
  $P^a_f(\mathcal S, c^*)\left\{
    \begin{array}{lll}
 \max & \sum_{\ell\in L} a_\ell\\
      \mbox{s.t.}  &(x^0, f^0)\in \mathcal F(OD^0)\\
&      (x^s, f^s)\in \mathcal F(OD^s) & s\in\mathcal S\\
&       \mathcal NO(x^0, f^0)= c^*\\
      & f^s_\ell - f^0_\ell \leq \mbox{MF}( 1-a_\ell)  & s\in\mathcal S,~\ell\in L\\
      & -f^s_\ell + f^0_\ell \leq \mbox{MF}( 1-a_\ell)  & s\in\mathcal S,~\ell\in L\\
       & a_\ell\in\s{0, 1} & \ell\in L
    \end{array}
\right.$
\end{center}

\noindent In any optimal solution, $a_\ell$ is equal to one unless $( f^s_\ell - f^0_\ell)$ is different from zero for some scenario $s$.

\bigskip

\noindent \textbf{$k$-distance approach}\\

In this approach,  for each  scenario $s\in\mathcal S$, the number of variables  $f^0_\ell$ and $f^s_\ell$ which have different
values is limited to $k\in\mathbb N$ and the nominal objective of the nominal solution $x^0$ is minimized. 
Variable $\delta^s_\ell$ is equal to $1$ if
and only if the values of  $f^0_\ell$ and $f^s_\ell$ are different and
$0$ otherwise. 

\begin{center}
  $P^d_f(\mathcal S, k)\left\{
    \begin{array}{lll}
 \min &  \mathcal NO(x^0, f^0)\\
      \mbox{s.t.}& (x^0, f^0)\in \mathcal F(OD^0)\\
&      (x^s, f^s)\in \mathcal F(OD^s) & s\in\mathcal S\\
& f^s_\ell - f^0_\ell \leq \mbox{MF}~ \delta^s_\ell  & s\in\mathcal S,~\ell\in L\\
& -f^s_\ell + f^0_\ell \leq \mbox{MF}~ \delta^s_\ell  & s\in\mathcal S,~\ell\in L\\
& \sum_{\ell\in L} \delta^s_\ell\leq k & s\in\mathcal S\\

      &  \delta^s_\ell \in\s{0, 1} &  s\in\mathcal S,~\ell\in L
    \end{array}
\right.$
\end{center}

\subsection{Solution cost based on lines deployment}
\label{sec:roblines}

{In this section, the differences between the solutions are measured through the lines deployment variables $x^0$ of the nominal solution and $\s{x^s}_{s\in\mathcal S}$ of the scenario solutions.}
\medskip

\noindent \textbf{Proactive approach}\\

{The  solution cost based on lines deployment corresponds to the distance in structure $d_{struct}$ from Definition~\ref{def:structure} applied to the lines deployment variables $x^0$ and $\s{x^s}_{s\in\mathcal S}$ (i.e., $\sum_{s\in \mathcal S}\sum_{\ell\in L} |\mathbb 1_{x^0_{\ell}>0}-\mathbb 1_{x^s_{\ell}>0}|$)}.
 We consider variable $dx_\ell^s$  equal to
{$|\mathbb 1_{x^s_\ell>0}-\mathbb 1_{x^0_\ell>0}|$} for each $\ell\in L$ and $s\in\mathcal S$.

\begin{center}
  $P^p_x(\mathcal S, c^*)\left\{
    \begin{array}{lll}
 \min & \sum_{s\in \mathcal S}\sum_{\ell\in L} dx^s_\ell\\
      \mbox{s.t.} & (x^0, f^0)\in \mathcal F(OD^0)\\
&      (x^s, f^s)\in \mathcal F(OD^s) & s\in\mathcal S\\
  &     \mathcal NO(x^0, f^0)= c^*\\
 &     x^s_\ell - x^0_\ell \leq dx^s_\ell & s\in\mathcal S,~\ell\in L\\
 &     -x^s_\ell + x^0_\ell \leq dx^s_\ell & s\in\mathcal S,~\ell\in L
    \end{array}
\right.$
\end{center}

\bigskip

\noindent \textbf{Anchored approach}\\

For each
line $\ell\in L$, the binary  variable  $a_\ell$ is equal to $1$ if and
only if $x^0_\ell$ is anchored i.e. $x^s_\ell=x^0_\ell$ $\forall s\in\mathcal S$.

\begin{center}
  $P^a_x(\mathcal S, c^*)\left\{
    \begin{array}{lll}
 \max & \sum_{\ell\in L} a_\ell\\
      \mbox{s.t.}  &(x^0, f^0)\in \mathcal F(OD^0)\\
&      (x^s, f^s)\in \mathcal F(OD^s) & s\in\mathcal S\\
&       \mathcal NO(x^0, f^0)= c^*\\
      &x^s_\ell - x^0_\ell \leq 1-a_\ell & s\in\mathcal S,~\ell\in L\\
      &- x^s_\ell + x^0_\ell \leq 1-a_\ell & s\in\mathcal S,~\ell\in L\\

 & a_\ell\in\s{0, 1} & \ell\in L

    \end{array}
\right.$
\end{center}


\noindent \textbf{$k$-distance approach}\\

Variable $\delta^s_\ell$ is equal to $1$ if
and only if the values of  $x^0_\ell$ and $x^s_\ell$ are different and
$0$ otherwise. 

\begin{center}
  $P^d_x(\mathcal S, k)\left\{
    \begin{array}{lll}
 \min &  \mathcal NO(x^0, f^0)\\
      \mbox{s.t.}& (x^0, f^0)\in \mathcal F(OD^0)\\
&      (x^s, f^s)\in \mathcal F(OD^s) & s\in\mathcal S\\
& x^s_\ell - x^0_\ell \leq \delta^s_\ell &  s\in\mathcal S,~\ell\in L\\
& -x^s_\ell + x^0_\ell \leq \delta^s_\ell &  s\in\mathcal S,~\ell\in L\\
& \sum_{\ell\in L} \delta^s_\ell\leq k & s\in\mathcal S\\

      & \delta_\ell^s\in\s{0, 1} &  s\in\mathcal S,~\ell\in L
    \end{array}
\right.$
\end{center}




\subsection{Numerical results}

We adapt an instance from data considered in~\cite{bussieck1997optimal}\footnote{\url{https://www.gams.com/latest/gamslib_ml/libhtml/gamslib_lop.html}} which provides an OD-matrix $OD^0$ and a network $\mathcal N=(\SetStations, E, d)$ with  $|V|=23$, $|E|=31$ and $|L|=210$. The set of lines $L$ contains one line $\ell_{ij}$ for each  pair of stations $i,j\in \SetStations$ such that $OD^0_{ij}>0$ or $OD^0_{ji}>0$. 


The maximal frequency MF is fixed to $6$ which corresponds to  $1$ train on the line each $10$ minutes when a time interval of $1$ hour is considered. The line opening costs $\s{K_\ell}_{\ell\in L}$ are deduced from {the costs provided in the instance} and have mean values of $1300$. A fixed frequency cost $\s{K'_\ell}_{\ell\in L}$ of 100 is considered. The capacity of each train is fixed to $C=200$.

A set of $10$ uncertain OD-matrices $\s{OD^1, ..., OD^{10}}$ is generated. To create a reasonable deviation from the nominal OD matrix $OD^0$, the value $OD_{ij}^s$ has a $20\%$ chance of being equal to $OD^0_{ij}$, otherwise it is uniformly drawn from  $[0.9\times OD^0_{ij}, 1.1\times OD^0_{ij}]$. Consequently, $OD^s_{ij}=0$ whenever $OD^0_{ij}=0$.


\subsubsection{Approaches comparison}

The results obtained for the robustness in frequencies and the robustness in lines deployment for the considered approaches are presented in Table~\ref{tab:freqApproaches} and~\ref{tab:ldApproaches}, respectively. All the problems have been solved to optimality in less than an hour.

 In these tables, three instances with the first 2, the first 6, and all 10 scenarios are considered and this number of scenarios is reported in the first column. The $k$-distance approach requires to fix the value of parameter $k\in\mathbb N$ which represents the maximal number of differences between the nominal solution and any scenario solution. Since there is no a priori relevant choice for  this parameter we set its value to $0$, $1$, $2$, $4$ and $10$. Consequently, seven lines are associated to each instance, one for the proactive approach, one for the anchored approach, and five for the $k$-distance approach. The proactive and the anchored approaches both ensure that the nominal objective of the nominal solution is equal to its optimal value $c^*$. This value is obtained by the resolution a priori of the nominal problem $\overline P$. {Note that $\overline P$ is only solved once as $c^*=81742$ does not depend on the number of scenarios nor the solution cost.} 

 {For each line of Table~\ref{tab:freqApproaches}, a nominal solution $(x^0, f^0)$ as well as scenario solutions $\s{(x^s, f^s)}_{s\in\mathcal S}$ are obtained by solving $P^p_f(\mathcal S, c^*)$, $P^a_f(\mathcal S, c^*)$, or $P^k_f(\mathcal S, k)$. From these solutions, we compute a posteriori distance $d_{val}=\sum_{s\in\mathcal S}\sum_{\ell\in L} |f^s_\ell-f^0_\ell|$, the number of anchored lines frequency $\sum_{\ell\in L} \mathbb 1_{f^0=f^s\,    \forall s\in\mathcal S}$ and the nominal objective $\mathcal NO (x^0, f^0)$ through~\eqref{eq:deterministicObjective}.} The nominal objective is reported in the third column. {Since its optimality  is imposed by  the proactive and the anchored approach, in both cases it is equal to $c^*=81742$ regardless of the number of scenarios.} This objective is larger for the $k$-distance approach whenever $k$ is small. This highlights a drawback of the $k$-distance approach which does not allow  to constrain  the nominal objective of the nominal solution. In other applications, a low value of $k$ could even lead to infeasibilities.
 The number of differences in frequencies 
$d_{val}$  is reported in the fourth column of Table~\ref{tab:freqApproaches}. {The proactive approach computes a nominal solution of nominal objective $81742$ which additionally minimizes this number of differences ({e.g.}, $50$ for two scenarios).} The number of anchored  lines frequency  is reported in the fifth column. {Among all solutions of nominal objective value $81742$, the anchored approach returns one which maximizes the number of anchored lines (e.g., $200$ over $210$ for two scenarios)}. Table~\ref{tab:ldApproaches} contains similar results. Nominal and scenario solutions are computed by solving $P^p_x(\mathcal S, c^*)$, $P^a_x(\mathcal S, c^*)$, or $P^d_x(\mathcal S, k)$. The fourth column contains the number of differences in lines deployment  
$d_{struct}=\sum_{s\in\mathcal S}\sum_{\ell\in L} |x^s_\ell-x^0_\ell|$ while the fifth columns contains the number of lines which deployment is  anchored $\sum_{\ell\in L} \mathbb 1_{x^0=x^s\, \forall s\in\mathcal S}$. 

{The proactive and the anchored approaches do not lead to the same solutions as the solution cost is not optimal in the anchored approach and the number of anchored lines is not optimal in the proactive approach.}
These differences are increased with the number of scenarios as it becomes harder to satisfy both objectives simultaneously.  Note that the increase in {$d_{val}$ and $d_{struct}$}  provided by the anchored approach is often greater than the decrease in number of variables anchored  in the proactive approach. This is due to the fact that anchoring a variable is quite constraining since it requires its value to be identical in the nominal solution and in all the scenario solutions. The advantage is that it guarantees that parts of the nominal solution will not be disrupted. However, this comes at a price in terms of flexibility which is reflected by an increase of the {solution costs $d_{val}$ and $d_{struct}$}. In the $k$-distance approach, an increase of the nominal objective for low values of $k$ enables to obtain better {solution costs} than the proactive approach and better numbers of variables anchored than the anchored approach. {In particular, when $k=0$ the scenario solutions are necessarily identical to the nominal solution which leads to a solution cost of $0$ and an anchoring of all the $210$ lines.} However, we observe a quick deterioration of these two objectives when $k$ increases showing once again that the choice of parameter $k$ can be challenging. {Moreover, the value $k$ is shared by all scenarios while some of them may require less changes than others which can lead to less suitable scenario solutions.} 


\begin{table}[h]
\centering
\begin{tabular}{cllr@{}lr@{}l*{20}{l}}
\hline
\multirow{2}{*}{\textbf{$|\mathcal S|$}} & \multirow{2}{*}{\textbf{Method}} & \multirow{2}{*}{$\mathcal NO (x^0, f^0) $} & \multicolumn{2}{c}{\textbf{Lines frequency}}
& \multicolumn{2}{c}{\textbf{Number of}} & 
\\
 & & & \multicolumn{2}{c}{\textbf{differencies} $d_{val}$}
& \multicolumn{2}{c}{\textbf{anchored lines}} & 
\\
\hline
  
  \multirow{7}{*}{2} & Proactive & \textbf{81742} & \textbf{50}& & 196&~(-2\%) & \\
 & Anchored & {81742}& $\quad$76&~(+52\%) & $\quad$\textbf{200}& & \\
 & 0-distance & 89028~(+9\%) & 0&~(-100\%) & 210&~(+5\%) & \\
 & 1-distance & 86663~(+6\%) & 12&~(-76\%) & 208&~(+4\%) & \\
 & 2-distance & 84740~(+4\%) & 19&~(-62\%) & 206&~(+3\%) & \\
 & 4-distance & 82156~(+1\%) & 37&~(-26\%) & 202&~(+1\%) & \\
 & 10-distance & 81742 & 81&~(+62\%) & 196&~(-2\%) & \\
\hline
\multirow{7}{*}{6} & Proactive & \textbf{81742} & \textbf{120}& & 186&~(-7\%) & \\
 & Anchored & {81742} & 246&~(+105\%) & \textbf{199}& & \\
 & 0-distance & 91998~(+13\%) & 0&~(-100\%) & 210&~(+6\%) & \\
 & 1-distance & 87274~(+7\%) & 28&~(-77\%) & 204&~(+3\%) & \\
 & 2-distance & 85178~(+4\%) & 53&~(-56\%) & 200&~(+1\%) & \\
 & 4-distance & 82158~(+1\%) & 105&~(-12\%) & 193&~(-3\%) & \\
 & 10-distance & 81742 & 259&~(+116\%) & 176&~(-12\%) & \\
\hline
\multirow{7}{*}{10} & Proactive & \textbf{81742} & \textbf{196}& & 183&~(-8\%) & \\
 & Anchored & {81742} & 449&~(+129\%) & \textbf{198}&& \\
 & 0-distance & 94308~(+15\%) & 0&~(-100\%) & 210&~(+6\%) & \\
 & 1-distance & 89733~(+10\%) & 44&~(-78\%) & 201&~(+2\%) & \\
 & 2-distance & 86485~(+6\%) & 98&~(-50\%) & 200&~(+1\%) & \\
 & 4-distance & 83224~(+2\%) & 191&~(-3\%) & 187&~(-6\%) & \\
 & 10-distance & 81742 & 436&~(+122\%) & 161&~(-19\%) & \\
 \hline
\end{tabular}\caption{Result of each approach when considering the solution cost based on frequencies. {For a given number of scenarios $|\mathcal S|$, the percentage in a cell corresponds to the relative change between the cell value and the value in bold in the same column.}}
\label{tab:freqApproaches}
\end{table}

\begin{table}[h]
\centering
\begin{tabular}{cllr@{}lr@{}l*{20}{l}}
\hline
\multirow{2}{*}{\textbf{$|\mathcal S|$}} & \multirow{2}{*}{\textbf{Method}} & \multirow{2}{*}{$\mathcal NO (x^0, f^0) $} & \multicolumn{2}{c}{\textbf{Lines deployment}}
& \multicolumn{2}{c}{\textbf{Number of}} & 
\\
 & & & \multicolumn{2}{c}{\textbf{differencies} $d_{struct}$}
& \multicolumn{2}{c}{\textbf{anchored lines}} & 
\\
\hline

 \multirow{7}{*}{2} & Proactive & \textbf{81742} & $\quad$\textbf{10}&  & $\quad$200&~(-1\%) & \\
 & Anchored & {81742}& 13&~(+30\%) & \textbf{202}&& \\
 & 0-distance & 85709~(+5\%) & 0&~(-100\%) & 210&~(+4\%) & \\
 & 1-distance & 84486~(+3\%) & 2&~(-80\%) & 208&~(+3\%) & \\
 & 2-distance & 83340~(+2\%) & 4&~(-60\%) & 206&~(+2\%) & \\
 & 4-distance & 81950~($<$1\%) & 8&~(-20\%) & 203&~($<$1\%) & \\
 & 10-distance & 81742 & 17&~(+70\%) & 195&~(-3\%) & \\
\hline
\multirow{7}{*}{6} & Proactive & \textbf{81742} & \textbf{23}& & 196&~(-3\%) & \\
 & Anchored & {81742}& 40&~(+74\%) & \textbf{202}& & \\
 & 0-distance & 87598~(+7\%) & 0&~(-100\%) & 210&~(+4\%) & \\
 & 1-distance & 84774~(+4\%) & 6&~(-74\%) & 204&~(+1\%) & \\
 & 2-distance & 83621~(+2\%) & 11&~(-52\%) & 202& & \\
 & 4-distance & 81952~($<$1\%) & 22&~(-4\%) & 195&~(-3\%) & \\
 & 10-distance & 81742 & 53&~(+130\%) & 177&~(-12\%) & \\
\hline
\multirow{7}{*}{10} & Proactive & \textbf{81742}& \textbf{37}&& 190&~(-5\%) & \\
 & Anchored & {81742} & 69&~(+86\%) & \textbf{200}&& \\
 & 0-distance & 88907~(+9\%) & 0&~(-100\%) & 210&~(+5\%) & \\
 & 1-distance & 85921~(+5\%) & 10&~(-73\%) & 203&~(+2\%) & \\
 & 2-distance & 84045~(+3\%) & 17&~(-54\%) & 198&~(-1\%) & \\
 & 4-distance & 82634~(+1\%) & 37& & 191&~(-4\%) & \\
 & 10-distance & 81742 & 98&~(+165\%) & 165&~(-18\%) & \\
\hline
 
\end{tabular}\caption{Result of each approach when considering the solution cost based on lines deployment. {For a given number of scenarios $|\mathcal S|$, the percentage in a cell corresponds to the relative change between the cell value and the value in bold in the same column.}}
\label{tab:ldApproaches}
\end{table}

\subsubsection{Flexibility on the nominal objective}

Imposing an optimal  nominal objective value to the nominal solution of the proactive approach may be too restrictive as small increases of the nominal objective may lead to significant decreases of the solution cost. As presented in Equation~\eqref{eq:pro_price_eps}, a parameter $\varepsilon\geq 0$ which corresponds to an acceptable percentage of increase of the nominal objective can be introduced for this purpose. Consequently, the constraint in $P^p_f$ and $P^p_x$ which ensures that the nominal objective is equal to $c^*$ can be replaced by

\begin{equation}  
  \mathcal NO(x^0, f^0)\leq c^* (1+\varepsilon).
\end{equation}

Tables~\eqref{tab:freqEps} and~\eqref{tab:ldEps} present the results obtained for four values of $\varepsilon$ ($0\%$, $1\%$, $2\%$ and $5\%$) for the robustness in frequencies and the robustness in lines deployment, respectively. In both cases, allowing an increase of the nominal objective enables to significantly reduce the solution cost and increase the number of anchored variables. In particular, for $\varepsilon$ equals to $5\%$, the solution cost is at most equal to $1$ and the number of anchored variables is at least equal to $209$ over $210$ lines. 

Similar results have been observed in Table~\ref{tab:freqApproaches} and~\ref{tab:ldApproaches} with the recoverable robustness approach for small values of $k$. The major difference with the proactive approach is that it seems easier for a decision maker to chose an acceptable percentage~$\varepsilon$ of deterioration of the nominal objective  rather than a maximal number $k$ of differences allowed in terms of variables values. 

\begin{table}[h]
\begin{subtable}{0.45\linewidth}
    
 \centering\begin{tabular}{@{}cccc@{}*{20}{l}}
\hline\multirow{2}{*}{\textbf{$|\mathcal S|$}} & \multirow{2}{*}{$\mathbf{\varepsilon}$} & \multirow{2}{*}{$\mathcal NO (x^0, f^0) $} & \textbf{Lines frequency}
\\
& & & \textbf{differences $d_{val}$}\\
\hline
  
  \multirow{4}{*}{2} & 0\% & 81742 & 50 & \\
 & 1\% & 82557 & 29 & \\
 & 2\% & 83348 & 24 & \\
 & 5\% & 85797 & 11 & \\
\hline
\multirow{4}{*}{6} & 0\% & 81742 & 120 & \\
 & 1\% & 82558 & 71 & \\
 & 2\% & 83339 & 55 & \\
 & 5\% & 85789 & 32 & \\
\hline
\multirow{4}{*}{10} & 0\% & 81742 & 196 & \\
 & 1\% & 82558 & 121 & \\
 & 2\% & 83338 & 93 & \\
 & 5\% & 85790 & 59 & \\
\hline

\end{tabular}\caption{Solution cost based on the frequencies.}
\label{tab:freqEps}
\end{subtable}
\begin{subtable}{0.45\linewidth}
    
  \centering\begin{tabular}{@{}cccc@{}*{20}{l}}
  
\hline\multirow{2}{*}{\textbf{$|\mathcal S|$}} & \multirow{2}{*}{$\mathbf{\varepsilon}$} & \multirow{2}{*}{$\mathcal NO (x^0, f^0) $} & \textbf{Lines deployment}
\\
& & & \textbf{differences $d_{struct}$}\\
\hline
  
  \multirow{4}{*}{2} & 0\% & 81742 & 10 & \\
 & 1\% & 82448 & 6\\
 & 2\% & 83340 & 4\\
 & 5\% & 85822 & 0\\
\hline
\multirow{4}{*}{6} & 0\% & 81742 & 23 & \\
 & 1\% & 82349 & 16\\
 & 2\% & 83360 & 11\\
 & 5\% & 85809 & 2\\
\hline
\multirow{4}{*}{10} & 0\% & 81742 & 37 & \\
 & 1\% & 82449 & 25\\
 & 2\% & 83286 & 19\\
 & 5\% & 85734 & 6\\
\hline
\end{tabular}\caption{Solution cost based on the lines deployment.}
\label{tab:ldEps}
\end{subtable}
    \caption{Influence of parameter $\varepsilon$ of our proactive approach on the solution cost.}
    \label{tab:eps}
\end{table}

\subsection{Proactive and reactive approaches comparison}
\label{sec:pro_reac}

{The reactive approach is considered when the uncertainty has not been anticipated and the nominal solution $(x^0, f^0)$ is infeasible for a scenario $s$. In that context, the reactive problem $P^r$ provides a reactive solution $(x^r, f^r)$ feasible for $s$ and which solution cost with $(x^0, f^0)$ is minimal. Problem $P^r$ is similar to the proactive problem $P^p$ except that only one scenario $s\in \mathcal S$ is considered and that the nominal solution $(x^0, f^0)$ is given as an input rather than an output. The reactive line optimization problem associated with the solution cost based on frequencies can be modelled by considering a binary variable $d_\ell$ equal to $1$ if and only if $f^r_\ell$ is equal to $f^0_\ell$ for each $\ell\in L$:}
 
\begin{center}
  $P^r(s, x^0, f^0)\left\{
    \begin{array}{ll@{\hspace{-.1cm}}l}
 \min & \sum_{\ell\in L} d_\ell\\
      \mbox{s.t.}&      (x^r, f^r)\in \mathcal F(OD^s)\\
 &     f^r_\ell - f^0_\ell \leq d_\ell & \ell\in L\\
 &     -f^r_\ell + f^0_\ell \leq d_\ell & \ell\in L
    \end{array}
\right.$
\end{center}
\vspace{.3cm}

{We do not present the reactive model for the solution cost based on lines deployment as  it  only requires in the  last two sets of constraints to replace $f^r_\ell$ and $f^0_\ell$ by $x^r_\ell$ and $x^0_\ell$, respectively.}

{Let $Q^*$ be the set of nominal solutions with an optimal nominal objective $c^*$. The aim of our proactive approach can be viewed as finding a solution in $Q^*$ which minimizes the sum of the solution costs over the scenarios in set $\mathcal S$. To assess the efficiency of the proactive solution, we compare its solution cost to the ones of four other solutions from $Q^*$. These four solutions have been selected to be as different as possible  in order to be representative of $Q^*$ (see Appendix~\ref{sec:div_sol} for more details on the selection process).}

{Tables~\ref{tab:reacF} and~\ref{tab:reacX} present the results obtained when considering the solution cost based on frequencies and lines deployment, respectively. {For each considered nominal solution $(x, f)$, the solution cost $v(P^r(s, x, f))$ of any scenario $s\in\mathcal S$ is obtained by solving the reactive problem $P^r(s, x, f)$. The sum of the solution costs over all scenarios in $\mathcal S$ is then obtained by computing $\sum_{s\in\mathcal S} v(P^r(s, x, f))$ and this value is represented in the third column of both tables.}   The proactive solution necessarily returns the optimal solution cost and the others lead to a mean increase of $8\%$ with a maximal increase of $24\%$. This significant variability of the solution cost among solutions from $Q^*$ highlights the relevance of considering the proactive approach. In Table~\ref{tab:reacF}, the first reactive solution is very similar to the proactive solution as its solution cost is identical for $2$ and $6$ scenarios and it is only incremented for $10$ scenarios. In Table~\ref{tab:reacX} one reactive solution always leads to an optimal solution cost but it is not always the same depending on the number of scenarios. This shows that the choice of the scenarios is a sensitive task that may significantly impact the proactive solution.}

\begin{table}[h]
\begin{subtable}{0.45\linewidth}
\centering\begin{tabular}{cl@{}r@{}l@{}}
\hline
\multirow{2}{*}{$|\mathcal S|$} & \multirow{2}{*}{\textbf{Solution}} & \multicolumn{2}{c}{\textbf{Lines frequency}}\\
& & \multicolumn{2}{c}{\textbf{differences} $d_{val}$}\\\hline
\multirow{5}{*}{2} & Proactive & \textbf{50} & \\
 & Reactive $(x^A, f^A)$ & $\qquad$50 & \\
 & Reactive $(x^B, f^B)$ & 57&~(+14\%) \\
 & Reactive $(x^C, f^C)$ & 62&~(+24\%)  \\
 & Reactive $(x^D, f^D)$ & 56&~(+12\%)  \\
\hline
\multirow{5}{*}{6} & Proactive & \textbf{120}  \\
 & Reactive $(x^A, f^A)$ & 120 & \\
 & Reactive $(x^B, f^B)$ & 123&~(+2\%)  \\
 & Reactive $(x^C, f^C)$ & 133&~(+11\%)  \\
 & Reactive $(x^D, f^D)$ & 121&~(+1\%)  \\
\hline
\multirow{5}{*}{10} & Proactive& \textbf{196} & \\
 & Reactive $(x^A, f^A)$ & 197&~(+1\%)  \\
 & Reactive $(x^B, f^B)$ & 211&~(+8\%)  \\
 & Reactive $(x^C, f^C)$ & 220&~(+12\%) \\
 & Reactive $(x^D, f^D)$ & 204&~(+4\%) \\
\hline
\end{tabular}\caption{Solution cost based on frequencies.}
\label{tab:reacF}
\end{subtable}
\begin{subtable}{0.45\linewidth}
\centering\begin{tabular}{cl@{}r@{}l@{}}
\hline
\multirow{2}{*}{$|\mathcal S|$} & \multirow{2}{*}{\textbf{Solution}} & \multicolumn{2}{c}{\textbf{Lines deployment}}\\
& & \multicolumn{2}{c}{\textbf{differences} $d_{struct}$}\\\hline
\multirow{5}{*}{2} & Proactive& \textbf{10} & \\
 & Reactive $(x^A, f^A)$& 10 & \\
 & Reactive $(x^B, f^B)$& $\qquad\quad$12&~(+20\%) \\
 & Reactive $(x^C, f^C)$& 11&~(+10\%) \\
 & Reactive $(x^D, f^D)$& 11&~(+10\%) \\
\hline
\multirow{5}{*}{6} & Proactive& \textbf{23} & \\
 & Reactive $(x^A, f^A)$& 26&~(+13\%) \\
 & Reactive $(x^B, f^B)$& 28&~(+22\%) \\
 & Reactive $(x^C, f^C)$& 24&~(+4\%)  \\
 & Reactive $(x^D, f^D)$& 23 & \\
\hline
\multirow{5}{*}{10} & Proactive& \textbf{37} & \\
 & Reactive $(x^A, f^A)$ & 38&~(+3\%) \\
 & Reactive $(x^B, f^B)$ & 43&~(+16\%) \\
 & Reactive $(x^C, f^C)$ & 39&~(+5\%)  \\
 & Reactive $(x^D, f^D)$ & 37 & \\
\hline
\end{tabular}\caption{Solution cost based on lines deployment.}
\label{tab:reacX}
\end{subtable}
\caption{Comparison of the proactive and the reactive approaches with four different optimal nominal solutions $(x^A, f^A)$, $(x^B, f^B)$, $(x^C, f^C)$ and $(x^D, f^D)$. {For each solution $(x^i, y^i)$ and each scenario $s\in\mathcal S$ the reactive problem $P^r(s, x^i, y^i)$ is solved. The values in the table correspond to the sum of the solution costs obtained over all scenarios. For a given number of scenarios $|\mathcal S|$, the percentage in a cell corresponds to the relative change between the cell value and the value in bold.}}
\end{table}



\section*{Conclusion}

{We introduced a new robust approach which optimizes the solution robustness by minimizing the solution cost over a discrete set of scenarios while ensuring the optimality of the nominal objective. We proved that the proactive counterparts of two polynomial network flow problems are NP-hard and that their reactive counterparts are polynomial for $d_{val}$ and NP-hard for $d_{struct}$.}

{We show in a case study of a railroad planning problem that a proactive solution can significantly reduce the solution cost compared to other solutions with an optimal nominal objective. Relaxing the optimality constraint on the nominal objective can also  further reduce the solution cost. Unlike the $k$-distance approach, the proactive approach does not require the definition of a parameter $k$ which may prove difficult to fix. We also observed that the anchored approach tends to increase more significantly the solution cost that the proactive approach decreases the number of anchored lines.}

{In future works it would be interesting to study the complexity of other problems or distances in this framework all the more so if it enables the identification of polynomial proactive integer problems. The discrete set of scenarios $\mathcal S$ could also be replaced by classical sets such as box,  budgeted or polytope uncertainty sets. This may lead to more challenging problems as it would no longer be possible to define a compact formulation which associates a set of variables to each possible scenario.  Finally, rather than imposing a bound on the nominal objective,  the proactive problem could be solved as a bi-objective problem in which both the solution cost and the nominal objective are minimized.}

\section*{Acknowledgement}

The beginnings of this work originates from Remi Lucas's PhD thesis~\cite{lucas2020planification}. The authors thank Fran\c cois Ramond, R\'emi Chevrier and R\'emi Lucas for the fruitful  collaboration during this thesis.

\clearpage
\section*{Appendix}
\appendix
\section{Nominal solutions generation}
\label{sec:div_sol}

{In Section~\ref{sec:pro_reac} we consider four solutions from $Q^*$, the set of solutions which have an optimal nominal objective $c^*$. We show that their solution cost vary significantly. We detail in this section how these solutions are obtained.}

{ In order to be as representative as possible of $Q^*$, we iteratively generate solutions which have a high solution cost with one another. More formally, if the solution cost is defined by a distance $d$ and if $Q\subset Q^*$ is the set of solutions obtained in previous iterations, we find a solution $x$ which distance is farthest from its nearest solution in $Q$: $x\in argmax_{x\in Q^*}\min_{x^q\in Q} d(x, x^q)$.}

{This problem can be solve through a MILP for the line optimization problem when considering any of the two distances presented in Section~\ref{sec:det_ol}. Since both models are similar we only present that for distance $d_{struct}$ over the lines deployment variables $x$.  In this model, variable $z$ represents the objective value while variables $x$ and $f$ correspond to the new solution. Finally, for each solution $x^q\in Q$ and each line $\ell\in L$, binary variable $dx^q_\ell$ is equal to $1$ if and only if $x_\ell$ is equal to $x^q_\ell$.}

\begin{center}$P_x(Q, c^*)\left\{
\begin{array}{lll@{\qquad}r}
 \max~ & z&\nonumber\\
     \mbox{s.t.~}& z\leq\sum_{\ell\in L} dx^q_\ell & x^q\in Q&\zcounter{eq:diff-z}\\
    ~& (x, f)\in \mathcal F(OD^0)&& \zcounter{eq:diff-xf}\\
~&       \mathcal NO(x, f)= c^*&&\zcounter{eq:diff-cstar}\\
 ~&     x^q_\ell - x_\ell \leq dx^q_\ell & x^q\in Q,~\ell\in L&\zcounter{eq:diff-xq1}\\
~    &     -x^q_\ell + x_\ell \leq dx^q_\ell & x^q\in Q,~\ell\in L&\zcounter{eq:diff-xq2}\\
\end{array}
\right.$
\end{center}

{Constraints~\zcount{eq:diff-z} ensures that the objective is lower than or equal to the solution cost between $x$ and any $x^q\in Q$. Constraints~\zcount{eq:diff-xf} and~\zcount{eq:diff-cstar} guarantee that $(x, f)\in Q^*$. Finally, the link between variables $x$ and $dx^q$ is ensured through constraints~\zcount{eq:diff-xq1} and~\zcount{eq:diff-xq2}.}

\bibliographystyle{plain}
\bibliography{bibliography}

\end{document}